\newif\ifpreprint
\preprinttrue 
\ifdefined\ispreprint
  \preprinttrue
\fi
\ifdefined\isjournal
  \preprintfalse
\fi
\RequirePackage{fix-cm}
\ifpreprint
  \documentclass{article}
  \usepackage[hidelinks]{hyperref}
  \usepackage{authblk}
  \usepackage[T1]{fontenc}
  \usepackage{lmodern}
  \usepackage{microtype}
  \usepackage[a4paper]{geometry}
  \usepackage[numbers,sort&compress]{natbib}
\else
  \documentclass[smallextended]{svjour3}
  \smartqed  
  \usepackage[T1]{fontenc}
  \usepackage{textcomp}
  \usepackage{url}
\fi

\usepackage{xcolor}
\usepackage{mathtools}
\usepackage{amsmath,amssymb,amsfonts,bm}

\allowdisplaybreaks

\newcommand{\R}{\mathbb{R}} 
\newcommand{\N}{\mathbb{N}} 
\newcommand{\CC}{\mathbb{C}}

\newcommand{\setu}{\mathfrak{u}}
\newcommand{\setv}{\mathfrak{v}}
\newcommand{\setU}{\mathfrak{U}}

\newcommand{\calW}{\mathcal{W}}

\newcommand{\rd}{\,\mathrm{d}} 
\newcommand{\rmd}{\mathrm{d}} 

\newcommand{\bszero}{{\boldsymbol{0}}} 
\newcommand{\bsb}{{\boldsymbol{b}}}    
\newcommand{\bse}{{\boldsymbol{e}}}    
\newcommand{\bsx}{{\boldsymbol{x}}}    
\newcommand{\bsy}{{\boldsymbol{y}}}    
\newcommand{\bsz}{{\boldsymbol{z}}}    
\newcommand{\bsDelta}{{\boldsymbol{\Delta}}}    
\newcommand{\bsalpha}{{\boldsymbol{\alpha}}}    
\newcommand{\bsgamma}{{\boldsymbol{\gamma}}}    
\newcommand{\bsnu}{{\boldsymbol{\nu}}}          
\newcommand{\bsomega}{\boldsymbol{\omega}}      
\newcommand{\bstau}{\boldsymbol{\tau}}			

\newcommand{\E}{\mathbb{E}} 

\newcommand{\e}{\mathrm{e}} 

\newcommand{\cost}{\mathrm{cost}}
\newcommand{\error}{\mathrm{error}} 

\DeclareMathOperator{\supp}{supp}

\DeclareMathAlphabet{\mymathbb}{U}{BOONDOX-ds}{m}{n}
\newcommand{\indicator}{\mymathbb{1}}

\newcommand{\essinf}{\mathop{\mathrm{ess\,inf}}}
\newcommand{\esssup}{\mathop{\mathrm{ess\,sup}}}

\def\citep#1#2{\cite[{#1}]{#2}}

\ifpreprint
\usepackage{amsthm}
\theoremstyle{plain}
  \newtheorem{theorem}{Theorem}
  \newtheorem{proposition}{Proposition}
  \newtheorem{lemma}{Lemma}
  
\theoremstyle{definition}

\theoremstyle{remark}
  \newtheorem{remark}{Remark}
\fi

\newcommand{\RefSec}[1]{Section~\textup{\ref{#1}}}

\newcommand{\RefThm}[1]{Theorem~\textup{\ref{#1}}}
\newcommand{\RefProp}[1]{Proposition~\textup{\ref{#1}}}
\newcommand{\RefLem}[1]{Lemma~\textup{\ref{#1}}}

\newcommand{\RefRem}[1]{Remark~\textup{\ref{#1}}}

\makeatletter
\renewcommand*{\@textcolor}[3]{%
  \protect\leavevmode
  \begingroup
    \color#1{#2}#3%
  \endgroup
}
\makeatother

\begin{document}

\title{
MDFEM: \ifpreprint \else \\ \fi Multivariate decomposition finite element method for
elliptic PDEs with uniform random diffusion coefficients \ifpreprint \\ \fi
using higher-order QMC and FEM
}

\ifpreprint
\else
\titlerunning{MDFEM for uniform diffusion coefficients}
\fi

\ifpreprint
  \author[1]{Dong T.\,P. Nguyen}
  \author[2]{Dirk Nuyens}
  \affil[1]{dong.nguyen@hcmut.edu.vn, Faculty of Computer Science and Engineering, \authorcr Ho Chi Minh City University of Technology, VNU-HCM, Vietnam}
  \affil[2]{dirk.nuyens@cs.kuleuven.be, Department of Computer Science, \authorcr KU Leuven, Celestijnenlaan 200A box 2402, B-3001 Leuven, Belgium}
  \date{\today}
\else
  \author{Dong T.\,P. Nguyen \and
          Dirk Nuyens}
  \institute{%
  Dong T.\,P. Nguyen \at Faculty of Computer Science and Engineering, Ho Chi Minh City University of Technology, VNU-HCM, Vietnam \\ \email{dong.nguyen@hcmut.edu.vn}
  \and
  Dirk Nuyens \at Department of Computer Science, KU Leuven, Celestijnenlaan 200A  box 2402, B-3001 Leuven, Belgium \\ \email{dirk.nuyens@cs.kuleuven.be}
  }
  \date{\today}
\fi

\maketitle

\begin{abstract}
We introduce the \emph{multivariate decomposition finite element method} (MDFEM) for solving elliptic PDEs with uniform random diffusion coefficients.
We show that the MDFEM can be used to reduce the computational complexity of estimating the expected value of a linear functional of the solution of the PDE.
The proposed algorithm combines the multivariate decomposition method (MDM), to compute infinite-dimensional integrals, with the finite element method (FEM), to solve different instances of the PDE.
The strategy of the MDFEM is to decompose the infinite-dimensional problem into multiple finite-dimensional ones which lends itself to easier parallelization than to solve a single large dimensional problem.
Our first result adjusts the analysis of the multivariate decomposition method to incorporate the $(\ln(n))^d$-factor which typically appears in error bounds for $d$-dimensional $n$-point cubature formulae and we take care of the fact that $n$ needs to come, e.g., in powers of~$2$ for higher order approximations.
For the further analysis we specialize the cubature methods to be two types of quasi-Monte Carlo (QMC) rules, being digitally shifted polynomial lattice rules and interlaced polynomial lattice rules.
The second and main contribution then presents a bound on the error of the MDFEM and shows higher-order convergence w.r.t.\ the total computational cost in case of the interlaced polynomial lattice rules in combination with a higher-order finite element method.
We show that the cost to achieve an error $\epsilon$ is of order $\epsilon^{-a_{\mathrm{MDFEM}}}$ with $a_{\mathrm{MDFEM}} = 1/\lambda + d'/\tau$ if the QMC cubature errors can be bounded by $n^{-\lambda}$ and the FE approximations converge like $h^\tau$ with cost $h^{d'}$, where $\lambda = \tau (1-p^*) / (p^* (1+d'/\tau))$ and $p^*$ is a parameter representing the ``sparsity'' of the random field expansion.
A comparison with a dimension truncation algorithm shows that the MDFEM will perform better than the truncation algorithm if $p^*$ is sufficiently small, i.e., the representation of the random field is sufficiently sparse.
\ifpreprint
\else
\fi
\end{abstract}

\ifpreprint
\fi

\section{Problem setting}\label{sec:introduction}

In this paper we propose and theoretically analyze the application of the \emph{multivariate decomposition method} (MDM) combined with the \emph{finite element method} (FEM)  to a class of elliptic PDEs with random diffusion coefficients.
We call the fusion of the two techniques the \emph{multivariate decomposition finite element method} or MDFEM in short.
Particularly, we consider a parametric elliptic Dirichlet problem
\begin{align}\label{eq:PDE}
  -\nabla \cdot (a(\bsx, \bsy) \, \nabla u(\bsx, \bsy))
  &=
  f(\bsx)
  ,
  &&\text{for $\bsx \in D$,}
\end{align}
with zero boundary condition,
for a domain $D \subset \R^d$, where usually $d=1$, $2$ or $3$,
and the gradient operator $\nabla$ is taken with respect to $\bsx$.
The parametric diffusion coefficient $a$ is assumed to depend linearly on the parameters $y_j$ as follows
\begin{align*}
  a(\bsx, \bsy)
  &=
  a_0(\bsx) + \sum_{j \ge 1} y_j \, \phi_j(\bsx)
  ,
  &
  y_j
  &\in
  \Omega := \left[-\tfrac12,\tfrac12\right]
  ,
\end{align*} 
for $\bsx \in D$ and the parameter vector $\bsy$ is distributed with the uniform probability measure on~$\Omega^\N$.
Here, $a_0$ is the mean field of $a$ and the fluctuations $\{\phi_j\}_{j \ge 1}$ are given functions.

The \emph{weak form} of the PDE is to find for given $\bsy \in \Omega^\N$ the solution $u(\cdot, \bsy) \in V := H_0^1(D)$ such that
\begin{align}\label{eq:PDE-weak-form}
  \int_D a(\bsx, \bsy) \, \nabla u(\bsx, \bsy) \cdot \nabla v(\bsx) \rd \bsx 
  =
  \int_D f(\bsx) \, v(\bsx) \rd \bsx
  ,
  \qquad
  \forall v \in V
  .
\end{align}

Our goal is to compute the expected value w.r.t.\ the parameter vector $\bsy \in \Omega^\N$ of a functional of the solution $u$ of the PDE.
That is, given a bounded linear functional $G : V \to \R$, we wish to compute the
integral
\begin{align}\label{eq:expectation-G}
  \mathbb{E}[G(u)]
  =
  I(G(u))
  &=
  \int_{\Omega^\N}
  G(u(\cdot,\bsy)) \rd \mu(\bsy)
  ,
\end{align}
with $\rmd \mu(\bsy) := \bigotimes_{j \ge 1} \rmd y_j$. This is an infinite-dimensional integral. 

Infinite-dimensional integration has been studied in a number of recent papers, see, e.g., \cite{CDMR2009,DKS2013,DG2014,DG2014b,GKNW2018,GMR2014,GHHR2017,KSWW2010b,KNPSW2017,PW2011,HMNR2010,DG2016}. Three kinds of algorithms have been introduced: \emph{single-level}, \emph{multi-level} and MDM, which is based on the earlier changing dimension algorithm. For an overview of single-, multi-level and the changing dimension algorithms we refer to~\cite[Section~7]{DKS2013} and the references therein.
In this paper we will consider the recently developed MDM.
The crucial idea of the MDM algorithm is to decompose the infinite-variate function into an infinite summation of functions depending only on a finite number of variables. This infinite summation is then truncated into a summation over a finite, so-called  \emph{active set} (of sets), and the infinite-dimensional integral is then wrapped into the sum and replaced by a specialized cubature rule in each case. 
The active set and cubature rules are selected in order to achieve an approximation up to a requested error while minimizing the computational cost. Particularly, to decompose the functions we will use \emph{the anchored decomposition method}, see, e.g.,~\cite{KSWW2010}. We will show that the decomposed functions belong to an \emph{anchored reproducing kernel Hilbert space} for which there exist (higher-order) deterministic or randomized quasi Monte-Carlo rules that can be used as cubature rules in the MDM algorithm. 

In order to approximate the infinite-dimensional integral~\eqref{eq:expectation-G} it is necessary to approximate the solution $u$. We use a FEM for this approximation.
Therefore, a spatial discretization error is added to the total error and the computational cost is now more expensive including the cost of the FEM compared to just approximating an infinite-dimensional integral of a given function. Based on an a priori error bound, the parameters of the MDFEM are chosen in order to achieve a prescribed accuracy by minimizing the computational work.
We prove in~\RefThm{thm:main-theorem} a combined error bound for the MDFEM which achieves higher-oder convergence w.r.t.\ the total computational cost in case of higher-order QMC rules in combination with higher-order FEM methods.

In our analysis the anchored decomposition of $G(u(\bsx,\bsy))$ with respect to the parametric variables $\bsy$ belongs to an \emph{infinite-variate weighted anchored reproducing kernel Hilbert space}. We find that under the condition of pointwise summability of the sequence $\{|\phi_j| \}_{j\ge 1}$, i.e., forthcoming condition~\eqref{eq:kappa}, exploiting the regularity of the solution $u$ with respect to $\bsy$, the weights which appear in the MDM analysis are product weights. 

Under the assumption that the diffusion coefficient $a$ is bounded away from zero and infinity, uniformly in the
parameter $\bsy$, the Lax--Milgram lemma ensures the existence and uniqueness of the solution $u$ of the weak problem~\eqref{eq:PDE-weak-form} in $V$. This leads us to make the following conditions on $a_0$ and $\{\phi_j\}_{j \ge 1}$.
We require that 
\begin{align}\label{eq:a0-max-min}
  a_0 \in L^{\infty}(D) \quad \text{ and } \quad \essinf_{\bsx \in D} a_0(\bsx) > 0
  .
\end{align}
Furthermore, we require the existence of a real-valued sequence $\{b_j\}_{j\ge 1}$, with $0 < b_j \le 1$ for all~$j$, and a constant $\kappa \in (0,1)$ such that
\begin{align}\label{eq:kappa}
    \kappa
    :=
	\left\| \frac{\sum_{j \ge 1} |\phi_j|/b_j}{2 a_0} \right\|_{L^\infty(D)}
	<\quad
	\frac{1}{2\alpha+1}
	\quad<\quad
	1
	,
\end{align}
for some $\alpha \in \N$.
To state our main result, \RefThm{thm:main-theorem}, we further need
\begin{align}\label{eq:bj-pstar-summable}
	\{b_j\}_{j\ge 1} \in  \ell^{p^*}(\N)
	,
\end{align}
for some $p^* \in (0,1)$.
These assumptions are standard and a similar restriction on $\kappa$ to obtain higher-order convergence was, e.g., also used in~\cite{BCDM2017,HS2019,K2017}.

The condition~\eqref{eq:a0-max-min} provides two constants $0 < a_{0,\min} \le a_{0,\max} < \infty$ such that for a.e.\ $\bsx \in D$ 
\begin{align*}
  a_{0,\min} \le a_0(\bsx) \le a_{0,\max} 
  .
\end{align*}
This together with~\eqref{eq:kappa} implies that for a.e.\ $\bsx \in D$ and any $\bsy \in \Omega^\N = [-1/2,1/2]^\N$
\begin{align}\label{eq:a-min}
  a(\bsx, \bsy)
  \ge
  a_0(\bsx) -\frac1{2} \sum_{j \ge 1}|\phi_j(\bsx)| 
  \ge (1 - \kappa) \, a_0(\bsx)
  \ge (1-\kappa) \, a_{0,\min}
  >
  0
\end{align}
and
\begin{align*}
  a(\bsx, \bsy)
  \le
  a_0(\bsx) + \frac1{2} \sum_{j \ge 1}|\phi_j(\bsx)| 
  \le
  (1 + \kappa) \, a_0(\bsx)
  <
  (1+\kappa) \, a_{0,\max}
  <
  \infty
  .
\end{align*}
Thus, due to the Lax--Milgram lemma, for all $f \in V^*$ and any $\bsy \in \Omega^\N$ there exists a unique solution $u(\cdot, \bsy) \in V$ of the weak problem~\eqref{eq:PDE-weak-form} and this solution is uniformly bounded with respect to $\bsy$, see also \cite[Theorem~3.1]{KSS2012} and the references therein, that is, for any $\bsy \in \Omega^\N$ we have
\begin{align}\label{eq:apriori-bound-u}
  \|u(\cdot, \bsy)\|_V \le \frac{\|f\|_{V^*}}{(1 - \kappa)\, a_{0,\min}}
  .
\end{align}
The specific form of condition~\eqref{eq:kappa} was stated in~\cite{BCM2017} and widely considered in~\cite{BCDS2017,BCDM2017,GHS2018,GHS2018-MLQMC,K2017} to benefit from the possible local support of the basis functions $\{\phi_j \}_{j\ge 1}$. 
Let us illustrate this and assume for the moment that the $\{\phi_j\}_{j\ge 1}$ are a system of wavelets obtained by scaling and translation from a finite number of mother wavelets, as was considered in~\cite{KSS2015,BCM2017,BCDS2017,GHS2018,GHS2018-MLQMC,K2017}, i.e., 
\begin{align*}
	\bigl\{\phi_j\bigr\}_{j\ge 1}
	=
	\bigl\{ \phi_{\ell,k}: \ell \ge 1, k \in J_{\ell} \bigr\}
	,
\end{align*}
where $\ell$ indicates the scale level, $k$ indicates the location index
and $J_{\ell}$ denotes the set of all location indices at level $\ell$.
In what follows we now identify the index $j$ with the corresponding tuple $(\ell,k)$.
The diffusion coefficient is then represented in the form
\begin{align*}
	a(\bsx, \bsy) = a_0(\bsx) + \sum_{\ell \ge 1} \sum_{k \in J_{\ell}} y_{\ell,k} \, \phi_{\ell,k}(\bsx)
	.
\end{align*}
Under the reasonable assumption that the wavelet system has at most $\eta$ overlapping basis functions at each level $\ell$ we can choose the sequence $\{b_{\ell,k}\}$ explicitly as follows, for some $c_\delta > 0$,
\begin{align*}
  b_{\ell,k}
  &=
  c_\delta \, \|\phi_{\ell,k}\|_{L^\infty(D)} \, \ell^{1+\delta}
  ,
  \qquad
  \forall \delta > 0
  ,
\end{align*}
i.e., we can basically take the $b_{\ell,k}$ to be proportional to $\|\phi_{\ell,k}\|_{L^\infty(D)}$.
It then follows from the finite support and finite overlap of $\eta$ functions on each level that
\begin{align*}
  \left\| \frac{\sum_{j \ge 1} |\phi_j|/b_j}{2 a_0} \right\|_{L^\infty(D)} 
  &=
  \left\| \frac{\sum_{\ell \ge 1} \sum_{k \in J_\ell} |\phi_{\ell,k}|/b_{\ell,k}}{2 a_0} \right\|_{L^\infty(D)} 
  \\
  &\le
   \frac{\eta}{2c_\delta \,a_{0,\min}}
   \sum_{\ell \ge 1}  \ell^{-(1+\delta)}
   = 
   \frac{\eta\,\zeta(1+\delta)}{2c_\delta \,a_{0,\min}}
   ,
\end{align*}
where $\zeta(\cdot)$ is the Zeta function.
The constant $c_\delta$ can now be chosen to satisfy $\kappa < 1$, or $\kappa < 1/(2\,\alpha+1)$ in case of higher-order convergence.
If $\{\phi_j\}_{j\ge 1}$ are pointwise normalized such that for some positive constants $\sigma$ and $\hat \alpha$ 
\begin{align*}
  \|\phi_{\ell,k}\|_{L^{\infty}(D)}
  =
  \sigma 2^{- \hat\alpha \ell }
  ,
\end{align*} 
and  there exists a fixed ordering of the wavelets from coarser to finer scale, that is, there exists a bijective mapping $j : \N\times\N \to \N$ such that $j^{-1}(\ell_1,k_1) \le j^{-1}(\ell_2,k_2)$ for any $ 1\le \ell_1 \le \ell_2$, $k_1$ and $k_2$, then we have
\begin{align*}
  b_{\ell,k}
  =
  c_\delta \,\sigma \, \ell^{1+\delta} 2^{- \hat\alpha \ell }
  .
\end{align*}
Such ordering guarantees $b_j \lesssim j^{-\hat \alpha/d} (\ln(j))^{1+\delta}$ which implies $\{b_j\} \in \ell^{p^*}(\N)$ for any $p^* > d/\hat{\alpha}$.

Further, condition~\eqref{eq:kappa} is used to establish an estimation on the mixed derivatives of the solution $u$ with respect to the parameter $\bsy$, see~\RefProp{prop:sum-k-V-norm}. This estimation might follow from the result of~\cite{BCM2017,GHS2018}, but in this paper we provide a different proof strategy which is inspired by~\cite{BCDM2017,K2017}. The proposed proof is simpler because we avoid defining a so called \emph{auxiliary problem} as in~\cite[Proof of Theorem~3.1]{BCM2017} and~\cite[Section~4]{GHS2018}, and work directly on the given PDEs. However, in order to receive simpler weights in the selection of the MDFEM active set, see~\eqref{eq:MDFEM:active-set}, we impose the additional condition $\kappa < 1/(2\alpha+1)$ in~\RefLem{lem:sum-nu-V-norm}, a similar restriction on $\kappa$ was also used in~\cite{BCDM2017,HS2019,K2017}.
Our analysis delivers similar bounds as those of, e.g., \cite{BCM2017,GHS2018}, but specialized to our decomposed functions $u_\setu$ which appear in the MDM decompostion~\eqref{eq:MDM-u}, and in particular will allow us to choose very simple product weights with $\gamma_j = b_j$ in our infinite-variate norm~\eqref{eq:infinite-variate-p-norm}.

In~\cite{BCM2017,GHS2018} it is shown  that the locality of the system $\{\phi_j\}_{j\ge 1}$ plays an important role in the representation of the diffusion coefficient. Firstly,~\cite{BCM2017} shows that it leads to improve the convergence rate of best $n$-term approximation in the sense that, with the same decay of $\|\phi_j\|_{L_{\infty}(D)}$ as $j\to \infty$, representing the diffusion coefficient using a locally supported system $\{\phi_j\}_{j\ge 1}$ gives a convergence rate of one half order higher than when using a globally supported system.
Secondly, in~\cite{GHS2018} the locality of the system $\{\phi_j\}_{j\ge 1}$ leads to product weights in the analysis of the cubature rules, which in turn enables to reduce the computational cost of constructing good QMC cubature rules.
In contrast, the weights used to construct good QMC rules in, e.g., \cite{KSS2012,DKLNS2014}, are ``product and order dependent'' (POD) weights and incur a higher construction cost.
Note however, that we assume the construction of the cubature methods to be an a priori cost since our finite-variate function spaces are unweighted.

Let $\N := \{1,2,\ldots\}$ and $\N_0 := \{0, 1, 2, \ldots\}$.
We introduce some standard notations for the function spaces on the physical domain needed for the FEM error bounds in \RefSec{sec:FE-discretization}.
For any $m \in \N$, the classical Sobolev space $H^m(D) \subseteq L^2(D)$ consists of all functions having weak derivatives of order less than or equal to $m$ in $L^2(D)$,
\begin{align*}
  H^m(D)
  :=
  \bigl\{
   v
   :
   D \to \CC
   :
   \partial^{\bsomega}_\bsx v \in L^2(D) \text{ for all } \bsomega \in \N_0^d \text{ with }|\bsomega| \le m
  \bigr\}
\end{align*}
with $\partial^{\bsomega}_\bsx := \partial^{|\bsomega|}/\prod_{j=1}^d \partial^{\omega_j}_{x_j}$ and $|\bsomega| := \sum_{j=1}^d |\omega_j|$.
We identify $H^0(D)$ with $L^2(D)$.
Let $H_0^m(D)$ denote the Sobolev space with homogeneous boundary condition
\begin{align*}
  H_0^m(D)
  :=
  \bigl\{ v \in H^m(D) : v|_{\partial D} = 0 \bigr\}
\end{align*}
and norm 
\begin{align*}
  \|v\|_{H_0^m(D)}
  :=
  \left(
    \int_D \sum_{|\bsomega| = m} \left|\partial^{\bsomega}_\bsx v(\bsx) \right|^2 \rd \bsx
  \right)^{1/2}
  =
  \left(
    \sum_{|\bsomega|= m} \|\partial^{\bsomega}_\bsx v\|_{L^2(D)}^2
  \right)^{1/2}
.
\end{align*}
Note that this is a norm due to the boundary condition.
For $m = 1$ we define a separate symbol
\begin{align*}
  V
  :=
  H_0^1(D)
  :=
  \bigl\{ v \in H^1(D) : v|_{\partial D} = 0 \bigr\}
\end{align*}
with norm given by 
\begin{align}\label{eq:V-norm}
  \|v\|_V
  :=
  \|v\|_{H_0^1(D)}
  :=
  \left(\int_D \sum_{j=1}^d |\partial_{x_j} v(\bsx)|^2 \rd \bsx\right)^{1/2}
  =
  \|\nabla v\|_{L^2(D)}
  .
\end{align}
For any $r > 0$ with $r \notin \N$ we set $r= [r]+\{r\}$, with $[r]$ the integer part of $r$ and $\{r\}$ the fractional part of $r$, we define the Sobolev--Slobodeckij space $H^r(D)$ as the space of functions in $H^{[r]}(D)$ such that the following Slobodeckij semi-norm is finite
\begin{align*}
  |v|_{H^{\{r\}}(D)}
  :=
  \left(
    \int_D\int_D \frac{|v(\bsx)-v(\bsz)|^2}{|\bsx-\bsz|^{2\{r\}+d}} \rd \bsx \rd \bsz 
  \right)^{1/2}
  <
  \infty
  ,
\end{align*}
and the norm for $H^r(D)$ given by
\begin{align*}
  \|v\|_{H^r(D)}
  :=
  \left(\|v\|_{H^{[r]}(D)}^2+ |v|_{H^{\{r\}}(D)}^2\right)^{1/2}
  .
\end{align*}

The dual of $H_0^r(D)$ with respect to the pivot space $L^2(D)$ is denoted by $H^{-r}(D) := (H_0^r(D))^*$. Roughly speaking the duality pairing is the extension of the $L^2(D)$ inner product to $H^{-r}(D) \times H_0^r(D)$, see~\cite[Chapter~2.9]{TW2009}.

In a similar fashion, for any real non-negative $t$ we define another Sobolev space consisting of all functions having weak derivatives of order less than or equal to $t$ in $L^\infty(D)$
\begin{align*}
  W^{t,\infty}(D)
  :=
  \bigl\{
    v : D \to \CC : \partial^{\bsomega}_\bsx v \in L^\infty(D) \text{ for all } \bsomega \in \N_0^d \text{ with }|\bsomega| \le t
  \bigr\}
 .
\end{align*}
The norm is given by
\begin{align*}
  \|v\|_{W^{t,\infty}(D)}
  :=
  \max_{0 \le |\bsomega| \le t}
  \esssup_{\bsx \in D} \left|\partial^{\bsomega}_\bsx v(\bsx)\right|
  .
\end{align*}

The outline of the rest of this paper is as follows. In~\RefSec{sec:MDFEM-outline} we give the key ideas of the MDFEM and describe the basic steps in the MDFEM algorithm.
In~\RefSec{sec:MDM} we introduce the general MDM for approximating infinite-dimensional integrals with the selection of the active set and cubature rules. We refine the analysis of~\cite{KNPSW2017} and consider
a more flexible form of the convergence rate such that we can easily plug in higher-order QMC rules later which need the number of points to be a power of~$2$.
In~\RefSec{sec:QMC} we then introduce a higher-order anchored Sobolev space and specialize
the cubature rules to be (interlaced)
polynomial lattice rules that can achieve
higher-order convergence rates in the introduced space.
\RefSec{sec:parametric-regularity} considers the regularity of the solution $u$ with respect to the parametric variable~$\bsy$.
We obtain a bound on the norm of the functional $G$ of the solution which we need for the error analysis. 
Finally in~\RefSec{sec:error-analysis} we analyze the error of the MDFEM. Based on a priori error estimates, we select the active set, the cubature rules and the finite element meshsizes for the MDFEM.
We present our main result in this section, it is show in~\RefThm{thm:main-theorem} that the computational cost to achieve an accuracy of order $\epsilon$ is of order $\epsilon^{-a_{\mathrm{MDFEM}}}$ where $a_{\mathrm{MDFEM}} = 1/\lambda + d'/\tau$
if the QMC cubature errors can be bounded by $n^{-\lambda}$ and the FE approximations converge like $h^\tau$ with $\lambda = \tau (1-p^*) / (p^* (1+d'/\tau))$, and with $p^*$ representing the ``sparsity'' of the random field expansion through~\eqref{eq:kappa} and~\eqref{eq:bj-pstar-summable}.
By comparing with a single-level method we show that the multivariate decomposition method ideas can be used to reduce the computational complexity.
\RefSec{sec:conclusion} presents some concluding remarks.

In this paper $P \lesssim Q$ means there exists a constant $C$ independent of all relevant parameters such that $P \le C \, Q$.
Both the cardinality of a set and the $\ell^1$ norm of a vector are denoted by $|\cdot|$ but it should be clear from the context whichever is meant.
Througout we interpret $0^0$ as~$1$.

\section{Outline of the MDFEM}\label{sec:MDFEM-outline}

In this section we will first give some useful definitions and then introduce the main idea of the MDFEM. 
For any $\bsy \in \Omega^\N$ and $\setu \subset \N$, with $|\setu| < \infty$, we let $\bsy_\setu \in \Omega^\N$
denote the vector such that $(\bsy_\setu)_j = y_j$ for $j \in \setu$ and $0$ otherwise, and let $u(\cdot, \bsy_\setu)$ denote the ``\emph{$\setu$-truncated solution}'' of \eqref{eq:PDE} with $\bsy=\bsy_\setu$, that is, the solution of the problem:
\begin{align}\label{eq:u-truncated-solution}
  -\nabla \cdot ( a(\bsx, \bsy_\setu) \, \nabla u(\bsx, \bsy_\setu) )
  =
  f(\bsx)
  \text{ for $\bsx$ in $D$,}
  \;\,
  u(\bsx, \bsy_\setu)
  =
  0
  \text{ for $\bsx \in \partial D$}
,
\end{align}
where $a(\bsx, \bsy_\setu) = a_0(\bsx)+ \sum_{j \in \setu} y_j \, \phi_j(\bsx)$.
To approximate the solution to the variational form for any $\bsy_\setu$ we use the FEM. Let us define a finite dimensional subspace $V^h \subset V$, where the $h > 0$ is to be specified below, but it should be understood that $V^h \subset V^{h'} \subset V$ for $h' < h$.
We will solve the variational problem on $V^h$. The finite element approximation of the variational formulation of the $\setu$-truncated problem denoted by $u^h(\cdot, \bsy_\setu)$ is then to find for given $\bsy_\setu$ the solution $u^h(\cdot, \bsy_\setu) \in V^h$ such that 
\begin{align*}
  \int_D a(\bsx, \bsy_\setu) \, \nabla u^h(\bsx, \bsy_\setu) \cdot \nabla v(\bsx) \rd \bsx 
  =
  \int_D f(\bsx) \, v(\bsx) \rd \bsx
  ,
  \qquad
  \forall v \in V^h
  .
\end{align*}

The MDM strategy is to decompose the full solution $u$ in the form
\begin{align}\label{eq:MDM-u}
  u(\cdot, \bsy)
  =
  \sum_{|\setu| < \infty} u_\setu(\cdot, \bsy_\setu)
  ,
\end{align}
where the sum is over all finite subsets $\setu \subset \N$, and
\begin{align}\label{eq:uu}
  u_\setu(\cdot, \bsy_\setu)
  :=
  \sum_{\setv \subseteq \setu}(-1)^{|\setu|-|\setv|} \, u(\cdot, \bsy_\setv)
  .
\end{align}
We want to stress that $u_\setu(\cdot, \bsy_\setu)$ and $u(\cdot, \bsy_\setu)$ are different and we can only approximately evaluate $u(\cdot, \bsy_\setu)$ directly by the FEM.
Such decomposition of $u$ is called the \emph{anchored decomposition with anchor at $0$}, whose definition enforces that $u_\setu(\cdot, \bsy_\setu) =0$ whenever $y_j=0$ for any $j \in \setu$, see, e.g., \cite{KSWW2010,NN2021}.

Let $u_\setu^{h_\setu}(\cdot, \bsy_\setu)$ denote the finite element approximation of $u_\setu(\cdot, \bsy_\setu)$ obtained by  summing up the FEM approximations $u^{h_\setu}(\bsx, \bsy_\setv)$, i.e., 
\begin{align}\label{eq:uu-hu}
  u_\setu^{h_\setu}(\bsx, \bsy_\setu)
  :=
  \sum_{\setv \subseteq \setu}(-1)^{|\setu|-|\setv|} \, u^{h_\setu}(\bsx, \bsy_\setv)
  .
\end{align} 
Note that we use the same $h_\setu$ for all $\setv \subseteq \setu$ to approximate $u^{h_\setu}(\bsx,\bsy_\setv)$.

Due to the linearity and boundedness of $G$, we have
\begin{align}\label{eq:G-MDM}
  G(u(\bsx, \bsy))
  =
  \sum_{|\setu| < \infty} G(u_\setu(\bsx, \bsy_\setu))
  .
\end{align}
Let us define
\begin{align*}
  I_\setu(G(u_\setu)) := \int_{\Omega_\setu} G(u_\setu(\cdot, \bsy_\setu)) \rd\mu_\setu (\bsy_\setu)
\end{align*}
where $\Omega_\setu := \Omega^{|\setu|}$ and $\rmd\mu_\setu(\bsy_\setu) := \bigotimes_{j \in \setu}  \rmd\mu(y_j)$.
Under some assumptions, which will be specified in~\RefRem{rem:Gu-well-defined}, the decomposition~\eqref{eq:G-MDM} is well-defined, moreover, we can interchange integral and sum to obtain
\begin{align*}
  I(G(u))
  &=
  \int_{\Omega^\N}\sum_{|\setu| < \infty} G(u_\setu(\cdot, \bsy_\setu)) \rd \mu(\bsy)
  =
  \sum_{|\setu| < \infty} \int_{\Omega^\N} G(u_\setu(\cdot, \bsy_\setu))\rd \mu(\bsy)
  \\
  &=
  \sum_{|\setu| < \infty} \int_{\Omega_\setu} G(u_\setu(\cdot, \bsy_\setu)) \rd\mu_\setu (\bsy_\setu)
  =
  \sum_{|\setu| < \infty} I_\setu(G(u_\setu))
  .
\end{align*}

Given a desired error $\epsilon>0$, the MDFEM will decide which subsets $\setu \subset \N$ to include in the \emph{active set} $\setU(\epsilon)$ to approximate the infinite MDM sum.
Next, for each $\setu \in \setU(\epsilon)$ the integral of $G(u_\setu)$ needs to be approximated.
The integral is therefore replaced by a cubature formula using $|\setu|$-dimensional cubature nodes $\bsy_\setu^{(k)}$, and for each such node we use~\eqref{eq:uu-hu} to sum up the FEM approximations to obtain $u_\setu^{h_\setu}(\bsx,\bsy_\setu^{(k)})$.
More specifically, the MDFEM approximates~\eqref{eq:expectation-G} by 
\begin{align}\label{eq:MDFEM}
  Q_\epsilon^{\mathrm{MDFEM}}(G(u))
  :=
  \sum_{\setu \in \setU(\epsilon)} Q_{\setu,n_\setu}(G(u_\setu^{h_\setu}))
\end{align}
with
\begin{align*}
  Q_{\setu,n_\setu}(G(u_\setu^{h_\setu})) 
  :=
  \sum_{k=0}^{n_\setu-1} w_\setu^{(k)} \, G(u_\setu^{h_\setu}(\cdot, \bsy_\setu^{(k)}))
  ,
\end{align*}
where $\{(\bsy_\setu^{(k)}, w_\setu^{(k)})\}_{k=0}^{n_\setu-1}$ are the cubature nodes and their respective weights for the cubature rule $Q_{\setu,n_\setu}$.
For every $\setu \in \setU(\epsilon)$ the number of cubature nodes $n_\setu$ and the FEM meshsizes $h_\setu$ are chosen to minimize the computational cost of the algorithm. 

\section{General MDM setting: infinite-dimensional integration}\label{sec:MDM}

In this section we will introduce the MDM which is developed for computing integrals over an infinite-dimensional product region. We propose an improved error analysis in comparison to~\cite{KNPSW2017}. This allows us to consider integrals with respect to more general probability measures and apply higher-order quasi-Monte Carlo rules as cubature rules.
We consider  
\begin{align*}
  I(F)
  :=
  \int_{\Omega^\N} F(\bsy) \rd\mu(\bsy)
  ,
\end{align*}
where $\mu$ is the countable product of a one-dimensional probability measure over $\Omega$, that is, $\rmd \mu(\bsy) := \bigotimes_{j \ge 1} \mu(\rmd y_j)$.
A typical example is when $\Omega$ is bounded and $\mu$ is the uniform probability measure over $\Omega^\N$ as is the case in our problem setup where $\Omega = \left[-\frac{1}{2}, \frac{1}{2} \right]$.
Another example is when $\Omega = \R$ and $\mu$ is a Gaussian product measure over $\R^\N$, see also \cite{NN2021}.

The starting point of the MDM is that the integrand $F$ is given as a sum of finite-variate functions
\begin{align*}
  F(\bsy)
  =
  \sum_{|\setu| < \infty} F_\setu(\bsy_\setu)
  ,
\end{align*}
where the functions $F_\setu$ depend only on $\bsy_\setu$ and belong to some tangible function space.
In this paper each $F_\setu$ belongs to a reproducing kernel Hilbert space $H(K_\setu)$ with reproducing kernel $K_\setu$ and norm denoted by $\|\cdot\|_{H(K_\setu)}$. 
Further, $F$ belongs to the infinite-dimensional function space $H_{\bsgamma,p}$, for $1 \le p \le \infty$,
\begin{align}\label{eq:infinite-variate-p-norm}
  \|F\|_{H_{\bsgamma,p}} 
  :=
  \left( \sum_{|\setu| < \infty} \left( \gamma_\setu^{-1} \, \|F_\setu\|_{H(K_\setu)} \right)^p \right)^{1/p}
  ,
\end{align}
if this norm is finite, and we assume the standard $\sup$-definition when $p = \infty$.
We denote with $q \ge 1$ the H\"older-conjugate of $p$ such that $1/p + 1/q = 1$.
In \RefSec{sec:QMC} we will specialize the $H(K_\setu)$ spaces to be anchored Sobolev spaces for anchored functions and then the infinite-variate norm~\eqref{eq:infinite-variate-p-norm}, for $p=2$, is the limit of the standard $s$-dimensional anchored Sobolev space from the QMC literature for $s \to \infty$, see, e.g., \cite{DG2014,KSWW2010,NN2021}.
The positive numbers $\gamma_\setu$ are called weights and indicate the importance of the different subspaces.
In this paper we restrict ourselves to the case when each $H(K_\setu)$ is the $|\setu|$-fold tensor product of a one dimensional function space
\begin{align*}
  K_\setu(\bsx_\setu, \bsy_\setu)
  =
  \prod_{j \in \setu} K(x_j,y_j)
  ,
\end{align*}
where $K$ is a one dimensional reproducing kernel for which
there exists a constant $M$ such that
\begin{align}\label{eq:RKHS-M}
  M
  &:=
  \int_{\Omega} (K(y,y))^{1/2} \rd \mu(y)
  <
  \infty
  ,
  \\ \notag
  M_\setu
  &:=
  M^{|\setu|}
  =
  \int_{\Omega_\setu} (K_\setu(\bsy_\setu,\bsy_\setu))^{1/2} \rd \mu_\setu(\bsy_\setu)
  ,
\end{align}
with $\Omega_\setu = \Omega^{|\setu|}$ and $\rmd{\mu_\setu(\bsy_\setu)} = \bigotimes_{j \in \setu} \rmd{\mu(y_j)}$.
For our analysis it is sufficient that the weights $\gamma_\setu$ appearing in~\eqref{eq:infinite-variate-p-norm} are ``product weights'' given by
\begin{align*}
  \gamma_\setu
  :=
  \prod_{j \in \setu} \gamma_j
\end{align*}
for a positive sequence $\{\gamma_j\}_{j\ge1}$.
We define the product over the empty set to equal~$1$.
We also assume that there is a $p^* \in (0,q)$ such that
\begin{align*}
  \{\gamma_j\}_{j\ge 1} \in \ell^{p^*}(\N)
  .
\end{align*}
Note that smaller $p^*$ implies faster decay of the weight sequence $\{\gamma_j\}_{j\ge1}$ and implies a problem which depends less on higher dimensions, see also \RefRem{rem:Stechkin}.

The following result, which is modified from \cite[Lemma~10]{KNPSW2017}, will be used in the further part.

\begin{lemma}\label{lem:sum-finite}
	Let $\{\gamma_j\}_{j\ge 1}$ be a non-negative sequence such that $\{\gamma_j\}_{j\ge 1} \in \ell^{p^*}(\N)$ for some $p^* >0$. Then for any $T > 0$, $p_1 < 1$ and $p_2 \ge p^*$, it holds
	\begin{align*}
	\sum_{|\setu| < \infty}  |\setu|^{p_1 |\setu|} \, T^{|\setu|} \, \prod_{j\in \setu} \gamma_j^{p_2}< \infty
	.
	\end{align*} 
\end{lemma}
\begin{proof}
	We have
	\begin{align*}
	&\sum_{|\setu| < \infty} |\setu|^{p_1|\setu|} \, T^{|\setu|} \prod_{j \in \setu} \gamma_j^{p_2}
	=
	\sum_{\ell=0}^{\infty} \ell^{p_1\ell} \, T^\ell \sum_{|\setu|=\ell} \prod_{j \in \setu} \gamma_j^{p_2}
	\\
	&
	=
		\sum_{\ell=0}^\infty \frac{\ell^{p_1 \ell} T^{\ell}}{\ell!}
		\sum_{j_1=1}^\infty  \gamma_{j_1}^{p_2} \!\! \sum_{j_1 \ne j_2=1}^\infty  \gamma_{j_2}^{p_2} \; \cdots\!\!\!\!\!\!\!\!\! \sum_{\{j_1,\ldots,j_{\ell-1}\} \not\ni j_\ell = 1}^\infty \!\!\! \gamma_{j_\ell}^{p_2}
	\le 
	\sum_{\ell=0}^{\infty} \frac{\ell^{p_1 \ell} T^{\ell}}{\ell!} \left(\sum_{j=1}^\infty \gamma_j^{p_2}\right)^{\!\ell}
	.
	\end{align*}
	The result follows from the ratio test: set $a_\ell = \ell^{p_1 \ell} T^\ell (\ell!)^{-1} (\sum_{j=1}^\infty \gamma_j^{p_2})^\ell$, then
	\begin{align*}
	\lim_{\ell \to \infty}\frac{a_{\ell+1}}{a_{\ell}} 
	= \lim_{\ell \to \infty} 
	\frac1{(\ell+1)^{1-p_1}} \left(1 + \frac1{\ell}\right)^{p_1 \ell} T \, \sum_{j=1}^\infty \gamma_j^{p_2}
	=
	0
	<
	1
	\end{align*}
	when $p_1 < 1$ and $p_2 \ge p^*$ such that $\sum_{j=1}^\infty \gamma_j^{p_2} < \infty$.
\end{proof}

We will now show that the infinite-dimensional integral can be written as a sum of the finite-dimensional integrals~$I_\setu$ on $H(K_\setu)$.
From~\eqref{eq:RKHS-M} we can deduce that the integration functional on $H(K_\setu)$ is bounded since for every $F_\setu \in H(K_\setu)$, using the reproducing property of $K_\setu$ and the Cauchy--Schwarz inequality, we have
\begin{align*}
  F_\setu(\bsy_\setu)
  &=
  \langle F_\setu, K_\setu(\bsy_\setu, \cdot) \rangle_{H(K_\setu)}
  \\
  &\le
  \|F_\setu\|_{H(K_\setu)} \, \|K_\setu(\bsy_\setu, \cdot)\|_{H(K_\setu)}
  =
  \|F_\setu\|_{H(K_\setu)} \, (K_\setu(\bsy_\setu, \bsy_\setu))^{1/2}
  ,
\end{align*}
from which it follows that
\begin{align}\label{eq:Iu-bounded}
  \notag
  I_\setu(F_\setu)
  &=
  \int_{\Omega_\setu} F_\setu(\bsy_\setu) \rd\mu_\setu(\bsy_\setu)
  \\
  &\le 
  \|F_\setu\|_{H(K_\setu)} \int_{\Omega_\setu} (K_\setu(\bsy_\setu, \bsy_\setu))^{1/2} \rd\mu_\setu(\bsy_\setu) 
  \notag
  \\
  &=
  \|F_\setu\|_{H(K_\setu)} 
  \prod_{j \in \setu} \int_{\Omega} (K(y_j,y_j))^{1/2} \rd \mu(y_j)
  \le
  \|F_\setu\|_{H(K_\setu)} \, M_\setu
  <
  \infty 
  .
\end{align}
We know that if
$
  \sum_{|\setu| < \infty} |I_\setu(F_\setu)|  
  <
  \infty
$
then by applying Fubini's theorem we can interchange integral and sum to obtain
\begin{align*}
  I(F)
  &=
  \int_{\Omega^\N} \sum_{|\setu| < \infty} F_\setu(\bsy_\setu) \rd \mu(\bsy) 
  =
  \sum_{|\setu| < \infty} \int_{\Omega^\N} F_\setu(\bsy_\setu) \rd \mu(\bsy) 
  \\
  &= 
  \sum_{|\setu| < \infty} \int_{\Omega_\setu} F_\setu(\bsy_\setu) \rd \mu_\setu(\bsy_\setu )
  =
  \sum_{|\setu| < \infty} I_\setu(F_\setu)
  .
\end{align*}
This means that we can separate $I(F)$ into the sum of finite-dimensional integrals.  Using~\eqref{eq:Iu-bounded} and H{\"o}lder's inequality we have
\begin{align*}
  I(F)
  =
  \sum_{|\setu| < \infty} |I_\setu(F_\setu)| 
  &\le
  \sum_{|\setu| < \infty} \|F_\setu\|_{H(K_\setu)} \, M_\setu 
  =
  \sum_{|\setu| < \infty} \gamma_\setu^{-1} \, \|F_\setu\|_{H(K_\setu)} \, \gamma_\setu \, M_\setu
  \notag
  \\
  &\le 
  \left( \sum_{|\setu| < \infty} \left( \gamma_\setu^{-1} \, \|F_\setu\|_{H(K_\setu)} \right)^p \right)^{1/p}
  \left( \sum_{|\setu| < \infty} \left(\gamma_\setu \, M_\setu\right)^q \right)^{1/q}
  \notag
  \\
  &=
  \|F\|_{H_{\bsgamma,p}} \left( \sum_{|\setu| < \infty} \left(\gamma_\setu \, M_\setu\right)^q \right)^{1/q}
  <
  \infty
  ,
\end{align*}
where the last factor can be bounded
for $1 \le q < \infty$ by applying \RefLem{lem:sum-finite} and using
$\{\gamma_j\}_{j\ge 1} \in \ell^{p^*}(\N) \subset \ell^{q}(\N)$.
For the case that $q=\infty$ we need that $\sup_{|\setu|<\infty} \prod_{j\in\setu} ( \gamma_j \, M ) < \infty$ which is also satisfied since $\gamma_j$ is $p^*$-summable.

For each subspace $H(K_\setu)$ we now need a cubature rule 
\begin{align}\label{eq:Qu}
  Q_{\setu,n_\setu} (F_\setu)
  := 
  \sum_{k=0}^{n_\setu-1} w_\setu^{(k)} \, F_\setu(\bsy_\setu^{(k)})
  ,
\end{align}
where $\{(\bsy_\setu^{(k)},w_\setu^{(k)})\}_{k=0}^{n_\setu-1}$ are the cubature nodes and their respective weights.
Without having specified the space yet, we will assume the cubature rule can achieve a convergence rate $\lambda$ for $F_\setu \in H(K_\setu)$, with $H(K_\setu)$ a function space with ``sufficient'' smoothness, to be specified later in \RefSec{sec:QMC}, in the form
\begin{align}\label{eq:Qu-rate}
	|I_\setu(F_\setu) - Q_{\setu,n_\setu} (F_\setu)| 
	\le 
	\|F_\setu\|_{H(K_\setu)} \, C_{\setu,\lambda} \, \max\left\{ 1, \frac{(\ln(n_\setu))^{\lambda_1 |\setu|}} {n_\setu^\lambda} \right\}
	,
\end{align}
where $C_{\setu,\lambda}$ is a positive constant that might depend on $\setu$ and $\lambda$ and the maximum is there for when $n_\setu$ is $0$ or~$1$.
Note that this is a typical error bound for QMC and sparse grid cubatures in a dominating mixed smoothness Sobolev function space, see, e.g., \cite{BG2004,KSS2011,DKS2013,DG2014,KNPSW2017,KN2016}, and references therein, where the rate~$\lambda$ can be stated independent of the number of dimensions.
We note that to get higher order convergence one normally will have to restrict the numbers $n_\setu$ to be e.g.\ powers of $2$, see~\cite{HKKN2012}. We will therefore take care to assure that $n_\setu$ is either $0$ or a power of~$2$.

The following result is our error bound for the MDM for infinite-dimensional integration and is slightly modified from~\cite[Section~4.1]{KN2016} to allow a wider class of cubature rules with convergence as in~\eqref{eq:Qu-rate}.
Note that we trade the $\ln(n_\setu)^{\lambda_1 |\setu|}$ from~\eqref{eq:Qu-rate} in the cubature error with a factor $|\setu|^{\lambda_1 |\setu|}$ in the combined cubature error~\eqref{eq:MDM-cubature-error} in the next proposition, which means that we will have to control the $(\ln(n_\setu)/|\setu|)^{\lambda_1 |\setu|}$ factor in our error bound~\eqref{eq:MDM-bound} later in \RefThm{thm:MDM}.

\begin{proposition}\label{prop:MDM-error}
  Let $F$ belong to the function space $H_{\bsgamma,p}$ with $1 \le p \le \infty$ and norm~\eqref{eq:infinite-variate-p-norm},
  and $\{\gamma_j\}_{j\ge 1} \in \ell^{p^*}(\N)$ for some $p^* \in (0,q]$ with $\frac1{p} + \frac1{q}=1$.
  If, for a given requested error tolerance $\epsilon > 0$, the active set $\setU(\epsilon)$ is selected such that
  \begin{align}\label{eq:MDM-truncation-error}
    \left( \sum_{\setu \notin \setU(\epsilon)} \left(\gamma_\setu \, M_\setu\right)^q \right)^{1/q}
    &\le 
    \frac{\epsilon}{2} 
    ,
  \end{align}
  and for all $\setu \in \setU(\epsilon)$ the numbers $n_\setu$ are chosen such that 
  \begin{align}\label{eq:MDM-cubature-error}
    \left( \sum_{\setu \in \setU(\epsilon)}  \left( \frac{\gamma_\setu \, C_{\setu,\lambda} \, |\setu|^{\lambda_1 |\setu|} }{\max\{1, n_\setu^\lambda\}} \right)^q \right)^{1/q}
    \le
    \frac{\epsilon}{2}
    ,
  \end{align}
  then it holds for the MDM algorithm
  \begin{align*}
    Q_\epsilon(F)
    :=
    \sum_{\setu \in \setU(\epsilon)} Q_{\setu,n_\setu} (F_\setu)
    =
    \sum_{\setu \in \setU(\epsilon)} \sum_{k=0}^{n_\setu-1} w_\setu^{(k)} F_\setu(\bsy_\setu^{(k)})
  \end{align*}
  based on cubature rules~\eqref{eq:Qu} with error bounds of the form~\eqref{eq:Qu-rate} that
  \begin{align}\label{eq:MDM-bound}
    |I(F) - Q_\epsilon(F)|
    &\le \epsilon \, \|F\|_{H_{\bsgamma,p}} 
      \max\left\{
        1,
        \max_{\setu \in \setU(\epsilon)}\left(\frac{\ln(n_\setu)}{|\setu|}\right)^{\lambda_1 |\setu|}
      \right\}
    .
  \end{align}	
\end{proposition}
\begin{proof}
The error of the MDM algorithm is split into two terms 
\begin{align*}
  |I(F) - Q_\epsilon(F)| 
  \le
  \sum_{\setu \notin \setU(\epsilon)} |I_\setu(F_\setu)| 
  +
  \sum_{\setu \in \setU(\epsilon)} |I_\setu(F_\setu) - Q_{\setu,n_\setu} (F_\setu)|
  .
\end{align*}
For the truncation error we obtain
\begin{align*}
  \sum_{\setu \notin \setU(\epsilon)} |I_\setu(F_\setu)| 
  &\le 
  \left( \sum_{\setu \notin \setU(\epsilon)} \left( \gamma_\setu^{-1} \, \|F_\setu\|_{H(K_\setu)} \right)^p \right)^{1/p}
  \left( \sum_{\setu \notin \setU(\epsilon)} \left(\gamma_\setu \, M_\setu\right)^q \right)^{1/q}
  \\
  &\le
  \|F\|_{H_{\bsgamma,p}} \left( \sum_{\setu \notin \setU(\epsilon)} \left(\gamma_\setu \, M_\setu\right)^q \right)^{1/q}
  .
\end{align*}
While for the cubature error we obtain
\begin{align*}
  &\sum_{\setu \in \setU(\epsilon)} |I_\setu(F_\setu) - Q_{\setu,n_\setu} (F_\setu)|
  \le 
  \sum_{\setu \in \setU(\epsilon)} \|F_\setu\|_{H(K_\setu)} \, C_{\setu,\lambda} \frac{\max\{1, (\ln(n_\setu))^{\lambda_1 |\setu|}\}}{\max\{1, n_\setu^\lambda\}}
  \\
  &\;\le
  \left( \sum_{\setu \in \setU(\epsilon)} \left( \gamma_\setu^{-1} \, \|F_\setu\|_{H(K_\setu)} \right)^p \right)^{1/p}
  \left( \sum_{\setu \in \setU(\epsilon)} \left( \gamma_\setu \, C_{\setu,\lambda} \frac{\max\{1, (\ln(n_\setu))^{\lambda_1 |\setu|}\}}{\max\{1, n_\setu^\lambda\}} \right)^q \right)^{1/q}
  \\
  &\;\le 
  \|F\|_{H_{\bsgamma,p}}
  \left( \max \left\{ 1, \max_{\setu \in \setU(\epsilon)} \left(\frac{\ln(n_\setu)}{|\setu|}\right)^{\lambda_1 |\setu|} \right\} \right)
  \left( \sum_{\setu \in \setU(\epsilon)} \left( \frac{\gamma_\setu \, C_{\setu,\lambda} \, |\setu|^{\lambda_1 |\setu|}}{\max\{1, n_\setu^\lambda\}} \right)^q \right)^{1/q}
  .
\end{align*}
Combining both parts, together with~\eqref{eq:MDM-truncation-error} and~\eqref{eq:MDM-cubature-error}, we obtain the claimed result.
\end{proof}

We define the cost of the MDM algorithm to be
\begin{align}\label{eq:MDM-total-cost}
  \cost(Q_\epsilon)
  :=
  \sum_{\setu \in \setU(\epsilon)} n_\setu \, \pounds_\setu
  ,
\end{align}
where $\pounds_\setu$ is the cost of evaluating $F_\setu(\bsy_\setu)$ for any $\bsy_\setu \in \Omega_\setu$.
In our setup, $\pounds_\setu$ will have to take into account that we obtain $F_\setu$, or, $u_\setu$, by the anchored decomposition, cf.~\eqref{eq:uu}, and we will estimate it by $2^{|\setu|} \, \$(|\setu|)$ with $\$(|\setu|)$ the cost of evaluating a $\setu$-truncated solution.
Note that we restrict our study to the case when $\pounds_\setu$ depends only on the cardinality of $\setu$.

\subsection{Selection of the active set}

For any $\rho \in (1, q/p^*]$ we define the active set as
\begin{align}\label{eq:active-set}
  \setU(\epsilon) =
  \setU(\epsilon, q, \rho)
  &:=
  \left\{
    \setu \subset \N
    :
    \left(\gamma_\setu \, M_\setu\right)^{(1-1/\rho)}
    >
    \frac{\epsilon/2}{\left( \sum_{|\setv| < \infty} (\gamma_\setv \, M_\setv)^{q/\rho} \right)^{1/q}}
  \right\}
  ,
\end{align}
with the $\sup$-definition for the norm on the right hand side of the inequality when $q=\infty$.
For $q < \infty$ we can use \RefLem{lem:sum-finite}, or direct calculation, to show that $\sum_{|\setv| < \infty} (\gamma_\setv \, M_\setv)^{q/\rho} < \infty$ for $\rho < q/p^*$. For $q=\infty$ we can allow arbitrarily large $\rho$. Taking $\rho=\infty$ seems natural for $q=\infty$ and then we have
\begin{align*}
  \setU(\epsilon, \infty, \infty)
  &=
  \big\{ \setu \subset \N : \gamma_\setu \, M_\setu > \epsilon / 2 \big\}
  .
\end{align*}

It can be easily verified that the definition of the active set assures that the truncation error~\eqref{eq:MDM-truncation-error} is bounded by~$\epsilon/2$.
The following proposition from~\cite[Theorem~2]{W2013} shows that the cardinality of the active set is polynomial in~$1/\epsilon$.

\begin{proposition}
	Given $\gamma_\setu = \prod_{j\in\setu} \gamma_j$ with $\{\gamma_j\} \in \ell^{p^*}(\N)$ for some $p^*\in (0,q)$, and with $M_\setu = M^{|\setu|}$, then
	for any $\epsilon > 0$, $\rho \in (1, q/p^*]$
	and $\setU(\epsilon, q, \rho)$ as defined in~\eqref{eq:active-set}, it holds for $1 \le q < \infty$ that
	\begin{align*}
	|\setU(\epsilon,q,\rho)| < \left(\frac{2}{\epsilon} \right)^{q/(\rho-1)} \left(\sum_{|\setu| < \infty} (\gamma_\setu \, M_\setu)^{q/\rho} \right)^{\rho /(\rho-1)}
	\lesssim 
	 \epsilon ^{-q/(\rho-1)} 
	.
	\end{align*} 
    For the case $q = \infty$ and with $\rho = \infty$, it holds 
	\begin{align*}
	|\setU(\epsilon,\infty,\infty)|
	<
	\left(\frac{2}{\epsilon}\right)^{p^*}
	\sum_{|\setu| < \infty} \left( \gamma_\setu \, M_\setu \right)^{p^*}
	\lesssim 
	\epsilon^{-p^*}
	.
	\end{align*}
\end{proposition}

\begin{remark}
\label{rem:choice-of-rho}
	For $q \ne \infty$ this proposition states that the cardinality of the active set is of order $\epsilon^{-q/(\rho-1)}$ so for a fixed $q$ the parameter $\rho$ should be chosen as large as possible, i.e., $\rho=q/p^*$, if the aim is to achieve the smallest active set.
	For $q = \infty$ this means taking $\rho = \infty$.
\end{remark}

The following result from~\cite[Lemma~1]{PW2011}, see also \cite{W2013}, asserts that the active set only consists of functions depending on a low number of variables. 

\begin{proposition}
\label{prop:max-dim}
  Given $\gamma_\setu = \prod_{j\in\setu} \gamma_j$ with $\{\gamma_j\} \in \ell^{p^*}(\N)$ for some $p^*\in (0,q)$, and $M_\setu = M^{|\setu|}$, then
  for any $\epsilon >0$, $\rho \in (1,q/p^*]$ and $\setU(\epsilon, q, \rho)$ defined in~\eqref{eq:active-set},
  it holds 
  \begin{align*}
    d(\epsilon,q,\rho)
    &:=
    \max_{\setu \in \setU(\epsilon,q,\rho)} |\setu| = O\left( \frac{\ln(\epsilon^{-1})}{\ln(\ln(\epsilon^{-1}))}\right) 
    =
    o(\ln(\epsilon^{-1}))
    ,
  \end{align*}
  as $\epsilon \to 0$.
\end{proposition}

\subsection{Selection of the MDM cubature rules}\label{sec:MDM:Lagrange}

The key idea of the MDM algorithm is to select cubature rules $Q_{\setu,n_\setu}$ for all $\setu \in \setU(\epsilon,q,\rho)$ such that the computational cost~\eqref{eq:MDM-total-cost} is minimized with respect to $n_\setu$ under the constraint~\eqref{eq:MDM-cubature-error}.
Instead of minimizing for $n_\setu$ directly we look for positive real numbers $k_\setu \in \R$ and then set
\begin{align}\label{eq:MDM:nu}
  n_\setu
  &=
  2^{\log_2(\lfloor k_\setu \rfloor)} \in \N_0
  .
\end{align}
This guarantees that our $n_\setu$ are either $0$ or a power of~$2$.
Note that with this choice $n_\setu \le k_\setu$, and hence
\begin{align}\label{eq:cost-bound-ku}
  \cost(Q_\epsilon)
  =
  \sum_{\setu \in \setU(\epsilon,q,\rho)} n_\setu \, \pounds_\setu
  &\;\le 
  \sum_{\setu \in \setU(\epsilon,q,\rho)} k_\setu \, \pounds_\setu
  .
\end{align}
At the same time we will use a constraint that is an upper bound on the actual error bound, which we will now show for the two cases $1 \le q < \infty$ and $q = \infty$.

For $1 \le q < \infty$ we look for positive real numbers $k_\setu \in \R$ which solve
\begin{align*}
  &\text{minimize } \sum_{\setu \in \setU(\epsilon,q,\rho)} k_\setu \, \pounds_\setu
  \\
  &\text{subject to }  \left( \sum_{\setu \in \setU(\epsilon,q,\rho)} \left( \frac{\gamma_\setu \, 2^\lambda \, C_{\setu,\lambda} \, |\setu|^{\lambda_1|\setu|}}{k_\setu^\lambda} \right)^q \right)^{1/q}
  =
  \frac{\epsilon}{2}
  .
\end{align*}
Our constraint is an upper bound in the following way
\begin{align*}
  \left( \sum_{\setu \in \setU(\epsilon,q,\rho)} \left( \frac{\gamma_\setu \, C_{\setu,\lambda} \, |\setu|^{\lambda_1 |\setu|}}{\max\{1, n_\setu^\lambda\}} \right)^q \right)^{1/q}
  &\le
  \left( \sum_{\setu \in \setU(\epsilon,q,\rho)} \left( \frac{\gamma_\setu \, 2^\lambda \, C_{\setu,\lambda} \, |\setu|^{\lambda_1 |\setu|}}{\max\{1, 2 \, n_\setu\}^\lambda} \right)^q \right)^{1/q}
  \\
  &\le
  \left( \sum_{\setu \in \setU(\epsilon,q,\rho)} \left( \frac{\gamma_\setu \, 2^\lambda \, C_{\setu,\lambda} \, |\setu|^{\lambda_1 |\setu|}}{k_\setu^\lambda} \right)^q \right)^{1/q}
  \le
  \frac{\epsilon}{2}
  ,
\end{align*}
since $k_\setu \le \max\{1, 2 \, n_\setu\}$.
This constrained minimization problem can be solved using the Lagrange multiplier method which leads to choose
\begin{align}\label{eq:ku-q-lt-infty}
  k_\setu
  =
  \left(\frac{\epsilon}{2}\right)^{-1/\lambda}
  \left(
    \frac{\gamma_\setu \, 2^\lambda \, C_{\setu,\lambda} \, |\setu|^{\lambda_1|\setu|}}{\pounds_\setu^{1/q}}
  \right)^{q/(q\lambda+1)}
  L_\epsilon^{1/(q\lambda)}
  ,
\end{align}
with
\begin{align}\label{eq:L-epsilon}
  L_\epsilon
  &:=
  \sum_{\setu \in \setU(\epsilon,q,\rho)}
      \pounds_\setu^{q\lambda /(q \lambda +1)} 
      \left( \gamma_\setu \, 2^\lambda \, C_{\setu,\lambda} \, |\setu|^{\lambda_1 |\setu|} \right)^{q/(q\lambda + 1)}
  .
\end{align}
For the special choice of $q = 1$ and $\lambda_1 = 0$ this agrees with the derivation in~\cite{KNPSW2017} with the modification that we here also guarantee that $n_\setu$ is a power of~$2$, which is needed to get higher order of convergence for our QMC rules, see~\cite{HKKN2012}.

For $q = \infty$ we demand for all $\setu \in \setU(\epsilon)$:
\begin{align*}
  \frac{\gamma_\setu \, C_{\setu,\lambda} \, |\setu|^{\lambda_1 |\setu|}}{\max\{1, n_\setu^\lambda\}}
  \le
  \frac{\gamma_\setu \, 2^\lambda \, C_{\setu,\lambda} \, |\setu|^{\lambda_1 |\setu|}}{\max\{1, 2 \, n_\setu\}^\lambda}
  \le
  \frac{\gamma_\setu \, 2^\lambda \, C_{\setu,\lambda} \, |\setu|^{\lambda_1 |\setu|}}{k_\setu^\lambda}
  \le
  \frac{\epsilon}{2}
  .
\end{align*}
Hence we choose
\begin{align}\label{eq:ku-q-eq-infty}
  k_\setu
  &=
  \left(\frac{\epsilon}{2}\right)^{-1/\lambda} \,
  \left( \gamma_\setu \, 2^\lambda \, C_{\setu,\lambda} \, |\setu|^{\lambda_1 |\setu|} \right)^{1/\lambda}
  .
\end{align}
Finally, combining the selection of the active set and cubature rules leads to our main  result on the convergence of the MDM for infinite-dimensional integration in the next theorem.

\begin{theorem}
\label{thm:MDM}
  Let $F$ belong to the function space $H_{\bsgamma,p}$ with $1 \le p \le \infty$ and norm~\eqref{eq:infinite-variate-p-norm},
  and $\{\gamma_j\}_{j\ge 1} \in \ell^{p^*}(\N)$ for some $p^* \in (0,q)$ with $\frac1{p}+ \frac1{q}=1$. If, for a given requested error tolerance $\epsilon>0$, the active set $\setU(\epsilon,q,\rho)$ is selected as in~\eqref{eq:active-set}
  for any $\rho \in (1, q/p^*]$, and if for all $\setu \in \setU(\epsilon,q,\rho)$ the numbers $n_\setu$ are chosen as in~\eqref{eq:MDM:nu}, then it holds for the MDM algorithm,
  \begin{align*}
    Q_\epsilon(F)
    =
    \sum_{\setu \in \setU(\epsilon,q,\rho)} Q_{\setu,n_\setu}(F_\setu)
    =
    \sum_{\setu \in \setU(\epsilon,q,\rho)} \sum_{k=0}^{n_\setu-1} w_\setu^{(k)} \, F_\setu(\bsy_\setu^{(k)})
  \end{align*}
  based on cubature rules~\eqref{eq:Qu} with convergence that can be expressed in the form~\eqref{eq:Qu-rate} with $\lambda \le 1/p^* - 1/q$ and $\lambda_1 < \lambda + 1/q$, and where
  $C_{\setu,\lambda}$ and $\pounds_\setu$ are at most exponential in $|\setu|$, that
  \begin{align*}
    |I(F) - Q_\epsilon(F)| 
    &\le 
    \epsilon \, \|F\|_{H_{\bsgamma,p}} 
    \max\left\{1,\max_{\setu \in \setU(\epsilon,q,\rho)} \left(\frac{\ln(n_\setu)}{|\setu|}\right)^{\lambda_1 |\setu|}  \right\}
    \\
    &=
    \|F\|_{H_{\bsgamma,p}} \, \epsilon^{1 -\delta(\epsilon)}
  \end{align*}
  where $\delta(\epsilon) = O(\ln(\ln(\ln(\epsilon^{-1})))/\ln(\ln(\epsilon^{-1}))) = o(1)$ as $\epsilon \to 0$.
  Furthermore, the computational cost is bounded by
  \begin{align*}
    \cost(Q_\epsilon) 
    &\lesssim
    \epsilon^{-1/\lambda}
    .
  \end{align*}
\end{theorem}
\begin{proof}
  We want to use that $n_\setu \le k_\setu \lesssim \epsilon^{-1/\lambda}$ with $k_\setu$ given by~\eqref{eq:ku-q-lt-infty} or~\eqref{eq:ku-q-eq-infty}. For $1 \le q < \infty$ we need to hence first show that $L_\epsilon$, given in~\eqref{eq:L-epsilon}, stays uniformly bounded when $\epsilon \to 0$. That is, we want
  \begin{align}\label{eq:L}
    L
    :=
    \lim_{\epsilon \to 0} L_\epsilon
    =
    \sum_{|\setu| < \infty}
      \pounds_\setu^{q\lambda /(q \lambda +1)} 
      \left( \gamma_\setu \, 2^\lambda \, C_{\setu,\lambda} \, |\setu|^{\lambda_1 |\setu|} \right)^{q/(q\lambda + 1)}
    <
    \infty
    .
  \end{align}
  We can use \RefLem{lem:sum-finite} to show that $L < \infty$ provided $\pounds_\setu$ and $C_{\setu,\lambda}$ are at most exponential in $|\setu|$, i.e., $\pounds_\setu^\lambda \, 2^\lambda \, C_{\setu,\lambda} \le T^{|\setu|}$ for some~$T$, and
	if the following conditions are satisfied:
	\begin{align*}
	  \lambda_1 q /(q\lambda+1) &< 1
	  &\text{and}&&
	  q/(q\lambda+1) &\ge p^*
	  .
	\end{align*}
	Using $b = a^{\ln(b)/ \ln(a)}$ and $\ln(n_\setu) \lesssim \ln(\epsilon^{-1/\lambda}) \lesssim \ln(\epsilon^{-1})$ we have
	\begin{align*}
	  \max_{\substack{\emptyset \ne \setu \in \setU(\epsilon,q,\rho) \\ \text{s.t.\ } n_\setu \ge 1}}
	  \left(\frac{\ln(n_\setu)}{|\setu|}\right)^{\lambda_1 |\setu|}
	  &=
	  \max_{\substack{\emptyset \ne \setu \in \setU(\epsilon,q,\rho) \\ \text{s.t.\ } n_\setu \ge 1}}
	  \epsilon^{- \lambda_1 \, (|\setu| / \ln(\epsilon^{-1})) \ln(\ln(n_\setu) / |\setu|)}
	  \\
	  &\le
	  \max_{\emptyset \ne \setu \in \setU(\epsilon,q,\rho)}
	  \epsilon^{- \lambda_1 \, (|\setu| / \ln(\epsilon^{-1})) \ln(\ln(\epsilon^{-1}) / |\setu|)}
	  .
	\end{align*}
	We now follow a similar reasoning as that in~\cite[page~513]{PW2011}.
	We look at the behaviour of the function $x \ln(1/x)$ with $x = |\setu| / \ln(\epsilon^{-1})$ which increases with $x$ for $0 < x \le \e^{-1}$.
	Since by \RefProp{prop:max-dim} we have $|\setu| \le d(\epsilon,q,\rho) \lesssim \ln(\epsilon^{-1})/\ln(\ln(\epsilon^{-1}))$, we can make $x = |\setu| / \ln(\epsilon^{-1}) \le d(\epsilon,q,\rho) / \ln(\epsilon^{-1}) \lesssim 1/\ln(\ln(\epsilon^{-1})) \le \e^{-1}$ by taking $\epsilon$ small enough.
	Hence we have for $\epsilon \to 0$
	\begin{align*}
	\frac{|\setu|}{\ln(\epsilon^{-1})}\ln\left(\frac{\ln(\epsilon^{-1})}{|\setu|}\right)
	\lesssim
	\frac{d(\epsilon,q,\rho)}{\ln(\epsilon^{-1})}\ln\left(\frac{\ln(\epsilon^{-1})}{d(\epsilon,q,\rho)}\right)
	\lesssim
	\frac{\ln(\ln(\ln(\epsilon^{-1})))}{\ln(\ln(\epsilon^{-1}))}
	.
	\end{align*}
	From here the first claim follows.
	
	Since $n_\setu \le k_\setu$ the computational cost~\eqref{eq:MDM-total-cost} can be bounded as in~\eqref{eq:cost-bound-ku}.
	Using~\eqref{eq:ku-q-lt-infty} for $1 \le q < \infty$ and~\eqref{eq:ku-q-eq-infty} for $q = \infty$
	we can now show $\cost(Q_\epsilon) \lesssim \epsilon^{-1/\lambda}$.
	For $1 \le q < \infty$ we use that $L$ as defined in~\eqref{eq:L} is bounded, while
	for $q = \infty$
    we have
	\begin{align*}
	  \cost(Q_\epsilon)
	  \le
	  \left(\frac{2}{\epsilon}\right)^{1/\lambda}
	  \sum_{\setu \in \setU(\epsilon,\infty,\rho)}
	    \pounds_\setu
	    \left(\gamma_\setu \, 2^\lambda \, C_{\setu,\lambda} \, |\setu|^{\lambda_1 |\setu|}\right)^{1/\lambda}
	  \lesssim
	  \epsilon^{-1/\lambda} 
	 ,
	\end{align*}
	and we can bound the sum for $\epsilon \to \infty$ similarly as $L_\epsilon$ above by making use of \RefLem{lem:sum-finite}, obtaining the conditions $\lambda_1 / \lambda < 1$ and $\lambda \le 1/p^*$.
\end{proof}

\begin{remark}
\label{rem:Stechkin}
We close this section with some remarks on~\RefThm{thm:MDM}. Since $\delta(\epsilon)= o(1)$ as $\epsilon \to 0$ this theorem implies that by using cubature rules with convergence rate higher than or equal to $\lambda = 1/p^* - 1/q$ we can achieve
\begin{align*}
  |I(F) - Q_\epsilon(F)|
  &\lesssim
  \cost(Q_\epsilon)^{-\lambda + o(1)}
  =
  \cost(Q_\epsilon)^{-a + o(1)}
  \qquad\text{with}\quad
  a = \frac{1}{p^{*}} - \frac{1}{q}
  .
\end{align*}
The convergence rate of the MDM is governed by the way we measure the norm of $F$ using the parameter $1 \le p \le \infty$ for the $p$-norm~\eqref{eq:infinite-variate-p-norm} with $1/p + 1/q = 1$, and the \emph{sparsity} of the sequence $\{\gamma_j\}_{j\ge1} \in \ell^{p^*}(\N)$ via the smallest possible parameter $p^* > 0$.
This concept of sparsity is what is used in the literature on best $n$-term approximation where Stechkin's lemma implies the same convergence rate of order $1/p^* - 1/q$, see, e.g., \cite{CDS2010,CDS2011}.
These results also match the exponent of tractability from the information-based complexity domain, see, e.g., \cite{W2013,DG2014}.
\end{remark}

\section{Function space and polynomial lattice rules}\label{sec:QMC}

In this section we introduce a reproducing kernel Hilbert space which is suitable for the MDM setting together with (interlaced) polynomial lattice rules which will be used as cubature rules in the MDM algorithm.
Here we will only provide some key results and we refer the reader to~\cite{DKPS2005,DG2014,DKLNS2014,NN2021} for more details.

Let $\alpha \ge 1$ be an integer. We first consider the one-dimensional reproducing kernel for an anchored Sobolev space of order~$\alpha$ of anchored functions with anchor at~$0$ over $\Omega = \left[-\frac1{2},\frac1{2}\right]$, see, e.g., \cite[Section~5]{DG2014}, \cite[Example~4.2]{KSWW2010} and \cite[Proposition~9]{NN2021} which amends the kernel as given in \cite{DG2014,KSWW2010},
\begin{align*}
    K_{\alpha,0}(x,y)
    &:=
    \sum_{r=1}^{\alpha-1} \frac{x^r}{r!}\frac{y^r}{r!}
    + \indicator_{\{xy > 0\}} \int_0^{\min\{|x|,|y|\}} \frac{(|x|-t)^{\alpha-1}}{(\alpha-1)!} \frac{(|y|-t)^{\alpha-1}}{(\alpha-1)!} \rd{t}
    .
\end{align*}
For $F, G \in H(K_{\alpha,0})$, seen as functions of $y \in \Omega$, the inner product of the corresponding reproducing kernel Hilbert space $H(K_{\alpha,0})$ is given by
\begin{align*}
  \langle F, G \rangle_{H(K_{\alpha,0})}
  :=
  \sum_{r=1}^{\alpha-1} (\partial^r_y F)(0) \, (\partial^r_y G)(0) + \int_{-\frac1{2}}^{\frac1{2}} (\partial^\alpha_y F)(y) \, (\partial^\alpha_y G)(y) \rd y
  ,
\end{align*}
with the norm $\|\cdot\|_{H(K_{\alpha,0})} := \sqrt{\langle \cdot, \cdot \rangle_{H(K_{\alpha,0})}}$. Note that all functions in the considered function space satisfy $F(0)=0$.

For multivariate functions, we assume that $F_\setu$ belongs to
$H(K_{\alpha,0,\setu})$ which is a tensor product space with the reproducing kernel defined by
\begin{align*}
  K_{\alpha,0,\setu}(\bsx _\setu, \bsy _\setu)
  :=
  \prod_{j\in \setu} K_{\alpha,0}(x_j, y_j)
  ,
\end{align*}
for $\bsx _\setu, \bsy _\setu \in \left[-\frac1{2},\frac1{2}\right]^{|\setu|}$. The corresponding norm is then given by 
\begin{align}\label{eq:norm-H-alpha-0-u}
  \|F_\setu\|_{H(K_{\alpha,0,\setu})}
  :=
  \left(
    \sum_{\setv \subseteq \setu}
    \sum_{\bstau_{\setu \setminus \setv} \in \{1:\alpha-1\}^{|\setu \setminus \setv|}}
    \int_{[-\frac1{2},\frac1{2}]^{|\setv|}} 
    \left|
      \left(\partial^{(\bsalpha_\setv,\bstau_{\setu \setminus \setv})}_{\bsy_\setu} F_\setu\right)(\bsy_\setv)
    \right|^2 \rd \bsy_\setv
  \right)^{1/2}
  ,
\end{align}
where $(\bsalpha_\setv,\bstau_{\setu \setminus \setv})$ denotes a combination of two sets, i.e., $\bsomega_\setu = (\bsalpha_\setv,\bstau_{\setu \setminus \setv})$ with $\omega_j = \alpha$ for $j \in \setv$ and $\omega_j = \tau_j$ for $j \in \setu \setminus \setv$.
Note that the derivative inside the norm is evaluated at $\bsy_\setv$ which means the argument takes the values of $y_j$ for $j \in \setv$ and $0$ otherwise, this is sometimes denoted by $[\bsy_\setv; \bszero]$ in other papers, but to not overload our notation in the next section we use this simplified form.
For notational convenience, when we introduce the Bochner norm~\eqref{eq:Bochner-norm} in the next section, we will denote $H(K_{\alpha,0,\setu})$ by $H_{\alpha,0,\setu}$.

For the function space $H(K_{\alpha,0})$ there exists a constant $M$ such that~\eqref{eq:RKHS-M} is satisfied. Indeed, we have
\begin{align}\label{eq:H-alpha-0:M}
  M
  =
  \int_{-1/2}^{1/2} (K_{\alpha,0}(y,y))^{1/2} \rd y 
  &\le
  \max_{x \in \left[-\frac1{2},\frac1{2}\right]} |K_{\alpha,0}(x,x)|^{1/2} 
  <
  \infty
  .
\end{align}

Next we need cubature rules which can provide higher order convergence for the anchored Sobolev space $H_{\alpha,0,\setu}$. Several choices exist in the literature and we state two of these methods in the next proposition. For $\alpha \ge 2$ we can obtain higher-order convergence by making use of interlaced polynomial lattice rules as in \cite{DKLNS2014} to obtain a convergence rate arbitrarily close to~$\alpha$. Since the function space in \cite{DKLNS2014} is different, we need an embedding result from \cite{DG2014} which we explain in the proof. For $\alpha = 1$ we resort to randomly digitally shifted polynomial lattice rules and achieve a convergence rate arbitrarily close to~$1$, see~\cite[Theorem~5.3]{DKPS2005}.
There exist software packages for using these cubature rules, see, e.g., \cite{MPS}, as well as for the construction of ``good generating vectors'' for such rules, see, e.g., \cite{QMC4PDE}.

\begin{theorem}\label{thm:QMC-convergence}
Let $F_\setu$ belong to the function space $H_{\alpha,0,\setu}$.
\begin{itemize}

\item When $\alpha = 1$ we can use randomly digitally shifted polynomial lattice rules in base~$2$, using $n_\setu = 2^{m_\setu}$ points, with $m_\setu \in \N$, such that the root-mean-square error over the digital shift is bounded as
\begin{align}\label{eq:RPLR}
  \sqrt{\E^{\bsDelta_\setu}\!\left[
    \left|I_\setu(F_\setu) - Q_{n_\setu}^{\bsDelta_\setu}(F_\setu)\right|^2
  \right]}
  &\le
  \frac{2^\lambda \, C_{1,\lambda}^{|\setu| \lambda}}{n_\setu^\lambda}
  \, \|F_\setu\|_{H_{1,0,\setu}}
  ,
  &&
  \forall \lambda \in [1/2,1)
  ,
\end{align}
and with the constant given by
\begin{align*}
  C_{1,\lambda}
  &:=
  \begin{cases*}
    \left(\frac{13}{12}\right)^{1/(2\lambda)} + \frac{1}{6}, & if $\lambda = \frac1{2}$, \\
    \left(\frac{13}{12}\right)^{1/(2\lambda)} + \frac1{3^{1/(2\lambda)} (2^{1/\lambda}-2)}, & if $\frac1{2} < \lambda < 1$.
  \end{cases*}
\end{align*}

\item When $\alpha \ge 2$ we can use interlaced polynomial lattice rules in base~$2$ with interlacing factor~$\alpha$, using $n_\setu = 2^{m_\setu}$ points, with $m_\setu \in \N$, such that the error is bounded as
\begin{align}\label{eq:IPLR}
  |I_\setu(F_\setu) - Q_{n_\setu}(F_\setu)|
  &\le
  \frac{4^\lambda\,C_{\alpha,\lambda}^{|\setu|\lambda}}{n_\setu^\lambda} \,
  \|F_\setu\|_{H_{\alpha,0,\setu}}
  ,
  &&
  \forall \lambda \in [1,\alpha)
  ,
\end{align}
with the constant given by
\begin{align*}
  C_{\alpha,\lambda}
  &:=
  2^{\alpha(\alpha-1)/2} \,
  \left( \alpha! \, \sqrt{\alpha} \, \left(\tfrac{3}{2}\right)^\alpha \left(\tfrac{5}{3}\right)^{\alpha-1} \right)
  \,
  \left(1 + \frac1{2^{\alpha/\lambda} - 2}\right)^\alpha
  .
\end{align*}

\end{itemize}
\end{theorem}
\begin{proof}
The result for $\alpha = 1$ can be found in~\cite[Theorem~5.3]{DKPS2005}. The result in that paper is for an anchored Sobolev space where the functions are not necessarily anchored as in our setup, but from \cite[third property of Lemma~1]{NN2021} follows that the norm~\eqref{eq:norm-H-alpha-0-u} could be written exactly like the norm of the space from \cite{DKPS2005} without changing its value.

For $\alpha \ge 2$ we make use of the continuous embedding of $H_{\alpha,0,\setu}$ into a function space based on Walsh functions $\calW_{\alpha,\setu}$, see, e.g., \cite{DG2014}.
Specifically, we have the following bound from \cite[Corollary~3]{DG2014}:
\begin{align*}
  \|f\|_{\calW_{\alpha,\setu}}
  &\le
  \left( \alpha! \, \sqrt{\alpha} \, \left(\tfrac{3}{2}\right)^\alpha \left(\tfrac{5}{3}\right)^{\alpha-1} \right)^{|\setu|} \, \|f\|_{H_{\alpha,0,\setu}}
  .
\end{align*}
Next we note that, although the result in \cite{DKLNS2014} is for a certain unanchored Sobolev space, and here we have an anchored Sobolev space, the construction of the interlaced polynomial lattice rules in \cite{DKLNS2014} happens for the space $\calW_{\alpha,\setu}$.
Hence the component-by-component construction error bound from \cite{DKLNS2014} for interlaced polynomial lattice rules also holds in our case, multiplied with the embedding constant.
Since our spaces are unweighted we can set all weights to~$1$, similar to what was done  in~\cite{NN2021} for the adaptation to a (different from the one in \cite{DKLNS2014}) unweighted unanchored Sobolev space.
From here the result follows.
\end{proof}

We end this section by noting that good polynomial lattice rules and interlaced polynomial lattice rules can be constructed by the fast component-by-component algorithm with a cost of $O(\alpha |\setu| n \ln(n))$, see \cite{NC2006b} and~\cite{DKLNS2014,NN2021}.
There are also similar error bounds for so-called higher-order polynomial lattice rules, but their construction cost is the much higher $O(\alpha |\setu| n_\setu^\alpha \ln(n_\setu))$, see~\cite{BDLNP2012}.
Since our spaces are unweighted, the constructed rules can be stored and used for any MDM/MDFEM algorithm.

\section{Parametric regularity of the PDE solution}\label{sec:parametric-regularity}

In this section we discuss bounds on derivatives with respect to the parametric variables $\bsy$ of the solution~$u$.
This is a key ingredient to show how the ``sparsity'' of the random field can be used to determine the regularity of the integrand function.

We first define the space $H_{\alpha,0,\setu}(\Omega_\setu;V)$ which is the Bochner version of the space $H_{\alpha,0,\setu}$ with norm~\eqref{eq:norm-H-alpha-0-u} and the $V$-norm~\eqref{eq:V-norm}, with the norm
\begin{align}\label{eq:Bochner-norm}
  \|u_\setu\|_{H_{\alpha,0,\setu}(\Omega_\setu;V)}
  &:=
  \left(
    \sum_{\setv \subseteq \setu}
    \sum_{\bstau_{\setu \setminus \setv} \in \{1:\alpha-1\}^{|\setu \setminus \setv|}}
    \int_{\Omega_\setv}
      \left\| 
        \left(\partial^{(\bsalpha_\setv,\bstau_{\setu \setminus \setv})}_{\bsy_\setu} u_\setu\right)(\cdot,\bsy_\setv)
      \right\|_V^2
    \rd \mu_\setv(\bsy_\setv)
  \right)^{1/2}
  ,
\end{align}
where inside the integral we evaluate $\partial^{(\bsalpha_\setv,\bstau_{\setu \setminus \setv})}_{\bsy_\setu} u_\setu$ in points $\bsy_\setv \in \Omega^\N$ with $\setv \subseteq \setu$, i.e., where $y_j = 0$ for $j \notin \setv$.
We show that there exists a bound for this norm  depending on the sequence $\{b_j\}_{j\ge 1}$ given in~\eqref{eq:kappa}. Thanks to that result we then obtain bounds for $\|G(u_\setu)\|_{H_{\alpha,0,\setu}}$ and $\|G(u_\setu^{h_\setu})\|_{H_{\alpha,0,\setu}}$ which will be used in the analysis of the MDFEM algorithm.
We now show how the norm of $u_\setu$ can be written in terms of the $\setu$-truncated solution by simply substituting $u(\cdot,\cdot_\setu)$ into~\eqref{eq:Bochner-norm}.

\begin{lemma}\label{lem:norm-of-uu-to-u}
	For any $\alpha \in \N$ and any $\setu \subset \N$  it holds
\begin{align*}
&\|u_\setu\|_{H_{\alpha,0,\setu}(\Omega_\setu;V)}
=
\|u(\cdot,\cdot_\setu)\|_{H_{\alpha,0,\setu}(\Omega_\setu;V)}
.
\end{align*}	
\end{lemma}
\begin{proof}
Using~\eqref{eq:uu} we have for $\setu \ne \emptyset$ and any $\bsomega_\setu \in \N^{|\setu|}$ and any $\bsy_\setu \in \Omega_\setu$
		\begin{align*}
		\left(\partial^{\bsomega_\setu}_{\bsy_\setu} u_\setu\right)(\cdot, \bsy_\setu) 
		=  
		\sum_{\setv \subseteq \setu}(-1)^{|\setu|-|\setv|} \left(\partial^{\bsomega_\setu}_{\bsy_\setu} u(\cdot, \cdot_\setv) \right)(\cdot, \bsy_\setu) 
		=
		\left(\partial^{\bsomega_\setu}_{\bsy_\setu} u(\cdot, \cdot_\setu) \right)(\cdot, \bsy_\setu)
		,
		\end{align*}
		where  we use the fact that  the partial derivative $\left(\partial^{\bsomega_\setu}_{\bsy_\setu} u(\cdot, \cdot_\setv) \right)(\cdot, \bsy_\setu) =0$ for all $\setv \subset \setu$, thus, the only surviving term is when $\setv = \setu$.
Using this together with~\eqref{eq:Bochner-norm} the result follows.
\end{proof}

To simplify further notation, for a given $\bsy \in \Omega^\N$, we introduce the energy norm
\begin{align*}
  \|v\|_{V,a_\bsy}
  &:=
  \sqrt{\int_D a(\bsx,\bsy) \, |\nabla v(\bsx)|^2 \rd \bsx}
  .
\end{align*}
Applying~\eqref{eq:a-min} it is easy to see that
\begin{align}\label{eq:lb-Va-norm}
  \sqrt{(1-\kappa) \, a_{0,\min}} \,\, \|v\|_V
  \le
  \|v\|_{V,a_\bsy}
  .
\end{align}
If we take $v(\bsx) = u(\bsx,\bsy)$ in~\eqref{eq:PDE-weak-form} and use the Cauchy--Schwarz inequality for the duality pairing, then we obtain
\begin{align*}
  \|u(\cdot, \bsy)\|_{V,a_\bsy}^2 = \int_D a(\bsx, \bsy) \, |\nabla u(\bsx, \bsy)|^2 \rd \bsx 
  =
  \int_D f(\bsx) \, u (\bsx, \bsy) \rd \bsx
  \le
  \|f\|_{V^*} \|u(\cdot, \bsy)\|_V
  .
\end{align*} 
Therefore, applying~\eqref{eq:lb-Va-norm} yields
\begin{align}\label{eq:ub-Va-norm}
  \|u(\cdot, \bsy)\|_{V,a_\bsy} 
  \le
  \frac{\|f\|_{V^*}}{\sqrt{(1 - \kappa)\,  a_{0,\min}}}
  .
\end{align}

For the next result we combine ideas from~\cite{BCDM2017,K2017}. Compared to \cite{BCM2017,GHS2018} we prove the regularity result directly without resorting to an auxiliary problem. The final result in \RefLem{lem:norm-uu-Guu} is similar to those in \cite{BCM2017,GHS2018}, but our results are specifically for $u_\setu$ and $G(u_\setu)$ and can be stated in a slightly simpler form. When we combine these bounds with the error analysis for the MDFEM in \RefSec{sec:error-analysis} they will lead to the particular simple choice of product weights with $\gamma_j = b_j$ for our infinite-variate function space.

\begin{proposition}\label{prop:sum-k-V-norm}
	Let $a_0 \in L^\infty(D)$ be such that $\essinf a_0 >0$, and there exists a sequence $\{b_j\}_{j\ge 1}$ with $0< b_j \le 1$ for all $j$, and a positive constant $\kappa\in (0,1)$ such that
	\begin{align*}
	\kappa
	=
	\left\| \frac{\sum_{j \ge 1} |\phi_j|/b_j}{2 \, a_0} \right\|_{L^\infty(D)}
	<1
	.
	\end{align*}  
	Then for any $\alpha \in \N$, $f \in V^*$, $\bsy\in \Omega^\N$ and any $k \in \N_0$  it holds
	\begin{align*}
	\sum_{\substack{|\bsnu|=k \\ \nu_j \le \alpha}}
	\bsb^{-2\bsnu}
	\|(\partial^\bsnu_\bsy u)(\cdot, \bsy) \|_V^2
	\le
	\left(
	\frac{2 \, \alpha\,\kappa}{1 - \kappa}
	\right)^{2k}
	\frac{\|f\|_{V^*}^2}{(1 -  \kappa)^2\, a_{0,\min}^2}
	,
	\end{align*}
where the sum is over $\bsnu \in \N_0^\N$ having only a finite number of nonzero indices, and we define $\bsb^\bsnu := \prod_{j \ge 1}b_j^{\nu_j}$.
\end{proposition}
\begin{proof} 
	We prove this result by induction on $\bsnu$. For $\bsnu = \bszero$ this is~\eqref{eq:apriori-bound-u}.	
	For $|\bsnu| \ge 1$ it is well-known that for any $\bsy \in \Omega^\N$, see, e.g.,~\cite{CDS2010,KN2016}, 
	\begin{align}\label{eq:Leibniz}
	\|(\partial^\bsnu_\bsy u)(\cdot, \bsy) \|_{V,a_\bsy}^2 
		=
		-
		\sum_{j \in \supp(\bsnu)} \nu_j \int_D \phi_j(\bsx) 
		\nabla (\partial^{\bsnu-\bse_j}_\bsy u)(\bsx,\bsy)
			\nabla (\partial^\bsnu_\bsy u)(\bsx,\bsy)
			\rd \bsx
			.
	\end{align}
Using~\eqref{eq:Leibniz} and then applying the Cauchy--Schwarz inequality to the sum over $j$ we get
	\begin{align*}
	&\sum_{\substack{|\bsnu|=k \\ \nu_j \le \alpha}} \bsb^{-2\bsnu} \|(\partial^\bsnu_\bsy u)(\cdot, \bsy) \|_{V,a_\bsy}^2 
		\\
	&\qquad=
	- 
	\int_D
	\sum_{\substack{|\bsnu|=k \\ \nu_j \le \alpha}}
	\sum_{j \in \supp(\bsnu)} 
	\bsb^{-\bse_j}\bsb^{-(\bsnu-\bse_j)} \bsb^{-\bsnu}
	\nu_j  \, \phi_j(\bsx) 
	\nabla (\partial^{\bsnu-\bse_j}_\bsy u)(\bsx,\bsy)
	\nabla (\partial^\bsnu_\bsy u)(\bsx,\bsy)
	\rd \bsx
	\\
	&\qquad\le
	\int_D
	\sum_{\substack{|\bsnu|=k \\ \nu_j \le \alpha}}
	\left(
	\sum_{j \in \supp(\bsnu)} 
		\bsb^{-\bse_j} 
		\nu_j  |\phi_j(\bsx)|   \left| \bsb^{-(\bsnu-\bse_j)} \nabla (\partial^{\bsnu-\bse_j}_\bsy u)(\bsx,\bsy)\right|^2
	\right)^{1/2}
	\\
	&
	\qquad\qquad\qquad
	\times
	\left(
	\sum_{j \in \supp(\bsnu)} 
	\bsb^{-\bse_j} \nu_j  |\phi_j(\bsx)|   \left| \bsb^{-\bsnu} \nabla (\partial^\bsnu_\bsy u)(\bsx,\bsy)\right|^2
	\right)^{1/2}
	\rd \bsx
	.
	\end{align*}
Again applying the Cauchy--Schwarz inequality to the sum over $\bsnu$ and to the integral over $D$ we have
	\begin{align}\label{eq:bound-sum-Va-norm}
	&\sum_{\substack{|\bsnu|=k \\ \nu_j \le \alpha}}
	\bsb^{-2\bsnu}
	\|(\partial^\bsnu_\bsy u)(\cdot, \bsy) \|_{V,a_\bsy}^2 
	\notag
	\\
	&\qquad\le
	\left(
	\int_D	
	\sum_{\substack{|\bsnu|=k \\ \nu_j \le \alpha}}
	\sum_{j \in \supp(\bsnu)} 
	\bsb^{-\bse_j} 
	\nu_j  |\phi_j(\bsx)|   \left| \bsb^{-(\bsnu-\bse_j)} \nabla (\partial^{\bsnu-\bse_j}_\bsy u)(\bsx,\bsy)\right|^2 \rd \bsx
	\right)^{1/2}
	\notag
	\\
	&
	\qquad\qquad\times
	\left(
	\int_D
	\sum_{\substack{|\bsnu|=k \\ \nu_j \le \alpha}}
	\sum_{j \in \supp(\bsnu)} 
	\bsb^{-\bse_j} \nu_j  |\phi_j(\bsx)|   \left| \bsb^{-\bsnu} \nabla (\partial^\bsnu_\bsy u)(\bsx,\bsy)\right|^2
		\rd \bsx
	\right)^{1/2}
	.
	\end{align}
Due to the fact that for $A_j \ge 0$ and $B_\bsnu \ge 0$
\begin{align*}
  \sum_{\substack{|\bsnu|=k \\ \nu_j \le \alpha}} \, \sum_{j \in \supp(\bsnu)} A_j B_{\bsnu - \bse_j}
  \le
  \sum_{\substack{|\bsnu|=k-1 \\ \nu_j \le \alpha}}\sum_{j \ge 1} A_j B_\bsnu
  =
  \left( \sum_{\substack{j \ge 1 \\ \vphantom{\le}}} A_j \right)
  \left( \sum_{\substack{|\bsnu|=k-1 \\ \nu_j \le \alpha}} B_\bsnu \right)
  ,
\end{align*}
which is equality without the condition $\nu_j \le \alpha$,
we write for the first factor in~\eqref{eq:bound-sum-Va-norm},
\begin{align}\label{eq:bound-Va-1}
&
\int_D
\sum_{\substack{|\bsnu|=k \\ \nu_j \le \alpha}}\,
\sum_{j \in \supp(\bsnu)} 
\bsb^{-\bse_j} \nu_j  |\phi_j(\bsx)|  
 \left| \bsb^{-(\bsnu-\bse_j)} \nabla (\partial^{\bsnu-\bse_j}_\bsy u)(\bsx,\bsy)\right|^2
\rd \bsx
\notag
 \\
&\qquad\le
\alpha
\int_D
\sum_{\substack{|\bsnu|=k \\ \nu_j \le \alpha}}\,
\sum_{j \in \supp(\bsnu)} 
\bsb^{-\bse_j} |\phi_j(\bsx)|  
\left| \bsb^{-(\bsnu-\bse_j)} \nabla (\partial^{\bsnu-\bse_j}_\bsy u)(\bsx,\bsy)\right|^2
\rd \bsx
\notag
\\
&\qquad\le
\alpha
\int_D
\sum_{j \ge 1} \bsb^{-\bse_j}  |\phi_j(\bsx)|  
\sum_{\substack{|\bsnu|=k-1 \\ \nu_j \le \alpha}} \left| \bsb^{-\bsnu} \nabla (\partial^\bsnu_\bsy u)(\bsx,\bsy)\right|^2
\rd \bsx
\notag
\\
&\qquad\le
\alpha
\left\|
\sum_{j \ge 1} 
\frac{ |\phi_j|/b_j}{a(\cdot,\bsy)}
\right\|_{L^\infty(D)}
\sum_{\substack{|\bsnu|=k-1 \\ \nu_j \le \alpha}}  \bsb^{-2\bsnu} \int_D a(\bsx, \bsy) \left| \nabla (\partial^\bsnu_\bsy u)(\bsx,\bsy)\right|^2
\rd \bsx
\notag
\\
&\qquad=
\alpha
\left\|
\sum_{j \ge 1} 
\frac{  |\phi_j|/b_j}{a(\cdot,\bsy)}
\right\|_{L^\infty(D)}
\sum_{\substack{|\bsnu|=k-1 \\ \nu_j \le \alpha}}  \bsb^{-2\bsnu} 
\|(\partial^\bsnu_\bsy u)(\cdot, \bsy) \|_{V,a_\bsy}^2 
.
\end{align}
Moreover, for the second factor in~\eqref{eq:bound-sum-Va-norm},
\begin{align}\label{eq:bound-Va-2}
&\int_D
\sum_{\substack{|\bsnu|=k \\ \nu_j \le \alpha}}
\sum_{j \in \supp(\bsnu)} 
\bsb^{-\bse_j} \nu_j  |\phi_j(\bsx)|   \left| \bsb^{-\bsnu} \nabla (\partial^\bsnu_\bsy u)(\bsx,\bsy)\right|^2
\rd \bsx
\notag
\\
&
\qquad\qquad\qquad
\le
\alpha
\left\|
\sum_{j \ge 1}
\frac{ |\phi_j|/b_j}{a(\cdot,\bsy)}
\right\|_{L^\infty(D)}
\sum_{\substack{|\bsnu|=k \\ \nu_j \le \alpha}}  \bsb^{-2\bsnu} 
\|(\partial^\bsnu_\bsy u)(\cdot, \bsy) \|_{V,a_\bsy}^2
.
\end{align}	
For any $\bsy \in \Omega^\N$ applying~\eqref{eq:a-min} and~\eqref{eq:kappa} we have
\begin{align}\label{eq:sum-to-kappa}
\left\|
\sum_{j \ge 1} 
\frac{  |\phi_j|/b_j}{a(\cdot,\bsy)}
\right\|_{L^\infty(D)}
\le 
\frac{1}{1 - \kappa}
\left\|
\sum_{j\ge 1} 
\frac{  |\phi_j|/b_j}{a_0}
\right\|_{L^\infty(D)}
=
\frac{2\,\kappa}{1 - \kappa}
.
\end{align}
Inserting~\eqref{eq:bound-Va-1},~\eqref{eq:bound-Va-2} and~\eqref{eq:sum-to-kappa} into~\eqref{eq:bound-sum-Va-norm} we have
\begin{align*}
&\sum_{\substack{|\bsnu|=k \\ \nu_j \le \alpha}}
\bsb^{-2\bsnu}
\|(\partial^\bsnu_\bsy u)(\cdot, \bsy) \|_{V,a_\bsy}^2 
\notag
\\
&\qquad\le
\frac{2\,\alpha\,\kappa}{1 - \kappa}
\left(
\sum_{\substack{|\bsnu|=k-1 \\ \nu_j \le \alpha}}  \bsb^{-2\bsnu} 
\|(\partial^\bsnu_\bsy u)(\cdot, \bsy) \|_{V,a_\bsy}^2 
\right)^{1/2}
\left(
\sum_{\substack{|\bsnu|=k \\ \nu_j \le \alpha}}
\bsb^{-2\bsnu}
\|(\partial^\bsnu_\bsy u)(\cdot, \bsy) \|_{V,a_\bsy}^2 
\right)^{1/2}
,
\end{align*}
and therefore
\begin{align*}
\sum_{\substack{|\bsnu|=k \\ \nu_j \le \alpha}}
\bsb^{-2\bsnu}
\|(\partial^\bsnu_\bsy u)(\cdot, \bsy) \|_{V,a_\bsy}^2 
&\le
\left(
\frac{2\,\alpha\,\kappa}{1 - \kappa}
\right)^2
\sum_{\substack{|\bsnu|=k-1 \\ \nu_j \le \alpha}}  \bsb^{-2\bsnu} 
\|(\partial^\bsnu_\bsy u)(\cdot, \bsy) \|_{V,a_\bsy}^2
.
\end{align*}
Using induction on $\bsnu$ we obtain 
\begin{align*}
\sum_{\substack{|\bsnu|=k \\ \nu_j \le \alpha}}
\bsb^{-2\bsnu}
\|(\partial^\bsnu_\bsy u)(\cdot, \bsy) \|_{V,a_\bsy}^2 
\le
\left(
\frac{2\,\alpha\,\kappa}{1 - \kappa}
\right)^{2k}
\|u(\cdot, \bsy) \|_{V,a_\bsy}^2
.
\end{align*}
Applying estimations~\eqref{eq:lb-Va-norm} and~\eqref{eq:ub-Va-norm} then implies
\begin{align*}
\sum_{\substack{|\bsnu|=k \\ \nu_j \le \alpha}}
\bsb^{-2\bsnu}
\|(\partial^\bsnu_\bsy u)(\cdot, \bsy) \|_V^2
\le
\left(
\frac{2\,\alpha\,\kappa}{1 - \kappa}
\right)^{2k}
 \frac{\|f\|_{V^*}^2}{(1 -  \kappa)^2 \, a_{0,\min}^2}
 ,
\end{align*}
which completes the proof.
\end{proof}

\begin{lemma}\label{lem:sum-nu-V-norm}
  For any $\alpha \in \N$, $\setu \subset \N$ and any $\bsy_\setu \in \Omega_\setu$ under the conditions of~\RefProp{prop:sum-k-V-norm} with 
  \begin{align*}
    \kappa
    <
    \frac1{2\alpha+1}
    ,
  \end{align*}
  it holds 
  \begin{align*}
    \sum_{\bsnu_\setu \in \{1:\alpha\}^{|\setu|}}
    \left\|
      \left(\partial^{\bsnu_\setu}_{\bsy_\setu}u(\cdot,\cdot_\setu)\right)(\cdot,\bsy_\setu)
    \right\|^2_V
    &\le
    \frac{C_{\kappa,\alpha}\|f\|_{V^*}^2}{(1 - \kappa)^2 a_{0,\min}^2}
    \prod_{j \in \setu} b_j^2
    ,
  \end{align*}
  where
  \begin{align*}
    C_{\kappa,\alpha}
    :=
    \sum_{k \ge 1} \left( \frac{2\,\alpha\,\kappa}{1 - \kappa} \right)^{2k}
    <
    \infty
    .
  \end{align*}
\end{lemma}
\begin{proof}
Note that $(\partial^{\bsnu_\setu}_{\bsy_\setu} u(\cdot,\cdot_\setu))(\cdot, \bsy_\setu) = (\partial^{\bsnu_\setu}_{\bsy_\setu} u)(\cdot, \bsy_\setu)$ since evaluating in $\bsy_\setu$ is setting all $y_j$ with $j \notin \setu$ to zero and hence it does not matter if we do this before or after taking partial derivatives w.r.t. components $y_j$ with $j \in \setu$.
It follows from~\RefProp{prop:sum-k-V-norm} that
\begin{align*}
  \sum_{\bsnu_\setu \in \{1:\alpha\}^{|\setu|}}
  \left[ \prod_{j \in \setu} b_j^{-2\nu_j} \right]
  \| (\partial^{\bsnu_\setu}_{\bsy_\setu} u(\cdot,\cdot_\setu))(\cdot, \bsy_\setu) \|_V^2
  &=
  \sum_{k \ge 1} \sum_{\substack{|\bsnu|=k \\ \supp(\bsnu) = \setu \\ \nu_j \le \alpha}}
  \bsb^{-2\bsnu} \| (\partial^{\bsnu_\setu}_{\bsy_\setu} u)(\cdot, \bsy_\setu) \|_V^2
  \\
  &\le
  \sum_{k \ge 1} \sum_{\substack{|\bsnu|=k \\ \nu_j \le \alpha}}
  \bsb^{-2\bsnu} \| (\partial^\bsnu_\bsy u)(\cdot, \bsy_\setu) \|_V^2
  \\
  &\le
  \sum_{k \ge 1} 
	\left(
	\frac{2\,\alpha\,\kappa}{1 - \kappa}
	\right)^{2k}
	\frac{\|f\|_{V^*}^2}{(1 -  \kappa)^2 \, a_{0,\min}^2}
  .
\end{align*}
Since $\kappa < \frac1{2\alpha+1}$, or equivalently $\frac{2\,\alpha\,\kappa}{1 - \kappa} < 1$, we have $C_{\kappa,\alpha} < \infty$.
Furthermore, since $0 < b_j \le 1$ for all $j$ we have that $\prod_{j \in \setu} b_j^{-2} \le \prod_{j \in \setu} b_j^{-2\nu_j}$ from which the claim follows.
\end{proof}

We can now show bounds on the norms of $u_\setu$ and $G(u_\setu)$.
Note that all arguments to show the regularity results in this section are all based on the weak formulation of the PDE. Since the weak formulation also holds when $V$ is replaced by $V^h \subset V$ the results hold true when the exact solution $u$ is replaced by its approximated solution $u^h$ with the constants independent of $h$, see, e.g., \cite{GKNSSS2015,KN2016,K2017}.

\begin{lemma}\label{lem:norm-uu-Guu}
  For any $\alpha \in \N$ and any $\setu \subset \N$ under the conditions of~\RefLem{lem:sum-nu-V-norm} it holds
  \ifpreprint 
  \begin{align*}
    \|u_\setu\|_{H_{\alpha,0,\setu}(\Omega_\setu;V)}
    \le
    \frac{C_{\kappa,\alpha}^{1/2}\ \|f\|_{V^*}}{(1 - \kappa)\, a_{0,\min}}
    \prod_{j \in \setu} b_j
    \quad\text{and}\quad
    \|u_\setu^{h_\setu}\|_{H_{\alpha,0,\setu}(\Omega_\setu;V)}
    \lesssim
    \frac{C_{\kappa,\alpha}^{1/2}\ \|f\|_{V^*}}{(1 - \kappa)\, a_{0,\min}}
    \prod_{j \in \setu} b_j
    .
  \end{align*}
  \else 
    \begin{align*}
    \|u_\setu\|_{H_{\alpha,0,\setu}(\Omega_\setu;V)}
    &\le
    \frac{C_{\kappa,\alpha}^{1/2}\ \|f\|_{V^*}}{(1 - \kappa)\, a_{0,\min}}
    \prod_{j \in \setu} b_j
    \intertext{and}
    \|u_\setu^{h_\setu}\|_{H_{\alpha,0,\setu}(\Omega_\setu;V)}
    &\lesssim
    \frac{C_{\kappa,\alpha}^{1/2}\ \|f\|_{V^*}}{(1 - \kappa)\, a_{0,\min}}
    \prod_{j \in \setu} b_j
    .
  \end{align*}
  \fi
  Furthermore, if $G \in V^*$ then
  \ifpreprint 
  \begin{align*}
	\|G(u_\setu)\|_{H_{\alpha,0,\setu}}
	\le
	\frac{C_{\kappa,\alpha}^{1/2} \ \|f\|_{V^*}\|G\|_{V^*}}{(1 - \kappa)\, a_{0,\min}}
	\prod_{j \in \setu} b_j
	\quad\text{and}\quad
	\|G(u_\setu^{h_\setu})\|_{H_{\alpha,0,\setu}}
	\lesssim
	\frac{C_{\kappa,\alpha}^{1/2} \ \|f\|_{V^*}\|G\|_{V^*}}{(1 - \kappa)\, a_{0,\min}}
	\prod_{j \in \setu} b_j
	.
  \end{align*}
  \else 
  \begin{align*}
	\|G(u_\setu)\|_{H_{\alpha,0,\setu}}
	&\le
	\frac{C_{\kappa,\alpha}^{1/2} \ \|f\|_{V^*}\|G\|_{V^*}}{(1 - \kappa)\, a_{0,\min}}
	\prod_{j \in \setu} b_j
	\intertext{and}
	\|G(u_\setu^{h_\setu})\|_{H_{\alpha,0,\setu}}
	&\lesssim
	\frac{C_{\kappa,\alpha}^{1/2} \ \|f\|_{V^*}\|G\|_{V^*}}{(1 - \kappa)\, a_{0,\min}}
	\prod_{j \in \setu} b_j
	.
  \end{align*}
  \fi
\end{lemma}
\begin{proof}
  By \RefLem{lem:norm-of-uu-to-u} and the definition~\eqref{eq:Bochner-norm} it is easy to see that 
  \begin{multline*}
    \|u(\cdot, \cdot_\setu)\|_{H_{\alpha,0,\setu}(\Omega_\setu;V)}^2
    \ifpreprint \else \\ \fi
    \le
    \sup_{\bsy_\setu \in \Omega_\setu}
    \sum_{\setv \subseteq \setu}
    \sum_{\bstau_{\setu \setminus \setv} \in \{1:\alpha-1\}^{|\setu \setminus \setv|}} 
    \left\| 
      \left(\partial^{(\bsalpha_\setv,\bstau_{\setu \setminus \setv})}_{\bsy_\setu} u(\cdot,\cdot_\setu)\right)(\cdot,\bsy_\setv,\bszero_{\setu\setminus \setv})
    \right\|_V^2 
    .
  \end{multline*}
  Applying~\RefLem{lem:sum-nu-V-norm} and taking the square root of the obtained inequality gives the first claim.

  Due to the linearity and boundedness of $G$ for any $\bsomega_\setu \in \N^{|\setu|}$ and $\bsy_\setu \in \Omega_\setu$, we have
  \begin{align*}
    \left(\partial_{\bsy_\setu} ^{\bsomega_\setu} G(u_\setu)\right)(\cdot, \bsy_\setu)
    &=
    \left(G\left(\partial_{\bsy_\setu} ^{\bsomega_\setu} u_\setu\right)\right)(\cdot, \bsy_\setu)
    \le 
    \|G\|_{V^*} 
    \left\|
      \left(\partial_{\bsy_\setu} ^{\bsomega_\setu} u_\setu\right)(\cdot, \bsy_\setu)
    \right\|_V
    .
  \end{align*}
  This proves the second claim.
\end{proof}

\section{Finite element discretization}\label{sec:FE-discretization}

In this section we briefly present the finite element method and  its error. The idea of the finite element method is to introduce a finite-dimensional subspace $V^h \subset V$ and solve the variational problem~\eqref{eq:PDE-weak-form} on $V^h$. Specifically, the domain $D$  is partitioned into \emph{elements}, e.g., subintervals, triangles or tetrahedrons with meshwidth $h > 0$ and $V^h$ is a set of polynomials that are defined piecewise on these elements and are globally continuous. The dimension of $V^h$
is of order $h^{-d}$, with $d$ denoting the spatial dimension.
The spaces and norms on the physical domain which we need here were introduced at the end of \RefSec{sec:introduction}.

We consider the case when the domain $D \subset \R^d$ is a convex and bounded polyhedron and
\begin{align*}
  f \in H^{-1+t}(D)
  \quad \text{ and } \quad
  G \in H^{-1+t'}(D)
  ,
\end{align*}
for some real parameters $t \ge 0$ and $t' \ge 0$.
In the case $0 \le t, t' \le 1$ we need the following condition on $a_0$ and $\{\phi_j\}_{j \ge 1}$: 
\begin{align}\label{eq:a-W1}
  a_0 \in W^{1,\infty}(D)
  \quad \text{ and } \quad
  \sum_{j \ge 1}\|\phi_j\|_{W^{1,\infty}(D)} < \infty
  ,
\end{align}
see \cite[Theorems~7.1 and~7.2]{KSS2012}.
In the case $t, t' > 1$, that is, when $f$ and $G$ have extra regularity, we need a stronger assumption.
More specifically, let $W_K^{t_0,\infty}(D)$
denote the \emph{weighted Sobolev space of Kondrat'ev type over $D$} with $t_0 := \max\{t,t'\}$, as defined in~\cite[Equation~(2.3)]{NS2013} and~\cite[Equation~(4.44)]{BCDS2017}. We then require
\begin{align}\label{eq:a-Wt}
  a_0 \in W_K^{t_0,\infty}(D)
  \quad \text{ and } \quad
  \sum_{j \ge 1}\|\phi_j\|_{W^{t_0,\infty}_K(D)} < \infty
  .
\end{align}   
Using higher-order FEMs it is then possible to achieve higher-order error bounds, see, e.g., to~\cite{GHS2018,DKLNS2014,KN2016} and~\cite[Assumption~4.1 and the proof of Lemma~4.1]{NS2013}.
Under these assumptions, for any $\bsy \in \Omega^\N$, we can use the bounds
\begin{align*}
  \left\| u(\cdot,\bsy) - u^h(\cdot,\bsy) \right\|_V
  \le
  C' \, h^t \, \|f\|_{H^{-1+t}(D)} 
\end{align*}
and
\begin{align}\label{eq:FE-error-Gu}
  \left| G(u(\cdot,\bsy)) - G(u^h(\cdot,\bsy)) \right|
  \le
  C \, h^\tau \, \|f\|_{H^{-1+t}(D)} \|G\|_{H^{-1+t'}(D)}
\end{align}
as $h \to 0$ with $\tau := t+t'$ and $C'$ and $C$ are constants independent of $h$ and~$\bsy$.

\section{Error and cost analysis of MDFEM: proof of main result}\label{sec:error-analysis}

In this section we give the main result of this paper which follows in~\RefThm{thm:main-theorem}.
As in \RefSec{sec:MDM} we will split the error, this time in a truncation error, a FE discretization error and a cubature error. In light of \RefThm{thm:QMC-convergence} we formulate this both for deterministic cubature rules and cubature rules which use a random element.

\subsection{Deterministic error bound}

We split the error of the MDFEM into three terms
\begin{multline*}
  I(G(u)) - Q_\epsilon^{\mathrm{MDFEM}}(G(u))
  =
  \left( I(G(u)) - \sum_{\setu \in \setU(\epsilon)} I_\setu(G(u_\setu)) \right)
  \\
  +
  \left( \sum_{\setu \in \setU(\epsilon)} I_\setu\left(G(u_\setu) - G(u_\setu^{h_\setu})\right) \right)
  +
  \left( \sum_{\setu \in \setU(\epsilon)} \left(I_\setu-Q_{\setu,n_\setu}\right)(G(u_\setu^{h_\setu}))\right)
  ,
\end{multline*}
which we will all bound individually.

The truncation error and the cubature error can be bounded in a similar way as in \RefProp{prop:MDM-error}, making use of the fact that $G(u_\setu)$ and $G(u_\setu^{h_\setu}) \in H_{\alpha,0,\setu}$ for any $\alpha \in \N$. The choice of $\alpha$ will be made later in this section and will be determined by the summability of the sequence $\{b_j\}_{j\ge1}$ as given in~\eqref{eq:kappa} and the choice of our weights $\gamma_j$ appearing in the norm~\eqref{eq:infinite-variate-p-norm}.

To bound the FE discretization error we use
\begin{align*}
  \left|
  \sum_{\setu \in \setU(\epsilon)} I_\setu\left(G(u_\setu) - G(u_\setu^{h_\setu})\right)
  \right|
  &\le
  \sum_{\setu \in \setU(\epsilon)} 
  \max_{\bsy_\setu \in \Omega_\setu} \left|G(u_\setu(\cdot, \bsy_\setu))- G(u_\setu^{h_\setu}(\cdot, \bsy_\setu))\right|
  .
\end{align*}
Moreover, using~\eqref{eq:uu}, the linearity of $G$ and~\eqref{eq:FE-error-Gu} we have for any $\bsy_\setu \in \Omega_\setu$ 
\begin{align*}
  \left|G(u_\setu(\cdot, \bsy_\setu))- G(u_\setu^{h_\setu}(\cdot, \bsy_\setu))\right|
  &=
  \left|
  \sum_{\setv \subseteq \setu}
  (-1)^{|\setu|-|\setv|} 
  \left( G(u(\cdot, \bsy_\setv)) - G(u^{h_\setu}(\cdot, \bsy_\setv)) \right)
  \right|
  \\
  &\le 
  \sum_{\setv \subseteq \setu}
  \left| G(u(\cdot, \bsy_\setv)) - G(u^{h_\setu}(\cdot, \bsy_\setv)) \right|
  \\
  &\le
  \sum_{\setv \subseteq \setu}
  C \, h_\setu^\tau \, \|f\|_{H^{-1+t}(D)} \, \|G\|_{H^{-1+t'}(D)}
  \\
  &=
  2^{|\setu|} \,
  C \, h_\setu^\tau \, \|f\|_{H^{-1+t}(D)} \, \|G\|_{H^{-1+t'}(D)}
  .
\end{align*}
Hence, we can bound the FE discretization error as
\begin{align*}
  \left|\sum_{\setu \in \setU(\epsilon)} 
    I_\setu\left(G(u_\setu) - G(u_\setu^{h_\setu})\right)
  \right|
  &\le
  \widetilde{C} \,
  \sum_{\setu \in \setU(\epsilon)} 2^{|\setu|} \, h_\setu^\tau
  ,
\end{align*}
where $\widetilde{C} = \widetilde{C}(f, G) := C \, \|f\|_{H^{-1+t}(D)} \, \|G\|_{H^{-1+t'}(D)}$.

To simplify the analysis we will pick $p=\infty$ for our infinite-variate norm~\eqref{eq:infinite-variate-p-norm}, and hence $q=1$, which means that in the next section we will be able to set $\gamma_j = b_j$ and consider the FE discretization errors and the cubature errors together. Hence we obtain
\ifpreprint 
\begin{multline}\label{eq:MDFEM-deterministic-bound}
  |I(G(u)) - Q_\epsilon^{\mathrm{MDFEM}}(G(u))|
  \le
  \left( \sup_{|\setu| < \infty} \gamma_\setu^{-1} \, \|G(u_\setu)\|_{H_{\alpha,0,\setu}} \right)
  \left( \sum_{\setu \notin \setU(\epsilon)} \gamma_\setu \, M_\setu \right)
  +
  \widetilde{C} \left( \sum_{\setu \in \setU(\epsilon)} 2^{|\setu|} \, h_\setu^\tau \right)
  \\
  +
  \left( \sup_{|\setu| < \infty} \gamma_\setu^{-1} \, \|G(u_\setu^{h_\setu})\|_{H_{\alpha,0,\setu}} \right)
  \max_{\setu \in \setU(\epsilon)} \left( \frac{\ln(n_\setu)}{|\setu|} \right)^{\lambda_1 |\setu|}
  \left( \sum_{\setu \in \setU(\epsilon)} \frac{\gamma_\setu \, C_{\setu,\lambda} \, |\setu|^{\lambda_1 |\setu|}}{n_\setu^\lambda} \right)
  .
\end{multline}
\else 
\begin{align}\label{eq:MDFEM-deterministic-bound}
  &|I(G(u)) - Q_\epsilon^{\mathrm{MDFEM}}(G(u))|
  \\\notag
  &\le
  \left( \sup_{|\setu| < \infty} \gamma_\setu^{-1} \, \|G(u_\setu)\|_{H_{\alpha,0,\setu}} \right)
  \left( \sum_{\setu \notin \setU(\epsilon)} \gamma_\setu \, M_\setu \right)
  +
  \widetilde{C} \left( \sum_{\setu \in \setU(\epsilon)} 2^{|\setu|} \, h_\setu^\tau \right)
  \\\notag
  &\;\;+
  \left( \sup_{|\setu| < \infty} \gamma_\setu^{-1} \, \|G(u_\setu^{h_\setu})\|_{H_{\alpha,0,\setu}} \right)
  \max_{\setu \in \setU(\epsilon)} \left( \frac{\ln(n_\setu)}{|\setu|} \right)^{\lambda_1 |\setu|}
  \left( \sum_{\setu \in \setU(\epsilon)} \frac{\gamma_\setu \, C_{\setu,\lambda} \, |\setu|^{\lambda_1 |\setu|}}{n_\setu^\lambda} \right)
  .
\end{align}
\fi

\subsection{Randomized error bound}

For cubature methods which include a random element we will bound the root-mean-square error over the random choices.
We require these cubature rules to be unbiased. This is true for the randomly digitally shifted polynomial lattice rules with error bound~\eqref{eq:RPLR}.
For the MDFEM we will analyse the case when the random elements $\bsDelta_\setu$ for each $Q_{\setu,n_\setu}^{\bsDelta_\setu}$ are independent of each other and we write the product expectation over all these independent random elements as~$\E_\bsDelta$.
Under these conditions the root-mean-square error of the MDFEM can be bounded as
\begin{align*}
  &\E_\bsDelta\left[\left| I(G(u)) - Q_\epsilon^{\mathrm{MDFEM}}(G(u)) \right|^2\right]
  \\
  &\ifpreprint \qquad \fi=
  \E_\bsDelta\left[\left|
    \left( I(G(u)) - \sum_{\setu \in \setU(\epsilon)} I_\setu(G(u_\setu^{h_\setu})) \right)
    + 
    \left( \sum_{\setu \in \setU(\epsilon)} \left(I_\setu-Q^{\bsDelta_\setu}_{\setu,n_\setu}\right)(G(u_\setu^{h_\setu})) \right)
  \right|^2\right]
  \\
  &\ifpreprint \qquad \fi=
  \left| I(G(u)) - \sum_{\setu \in \setU(\epsilon)} I_\setu(G(u_\setu^{h_\setu})) \right|^2
  +
  \E_\bsDelta\left[ 
  \left| \sum_{\setu \in \setU(\epsilon)} \left(I_\setu-Q^{\bsDelta_\setu}_{\setu,n_\setu}\right)(G(u_\setu^{h_\setu})) \right|^2
  \right]
  \\
  &\ifpreprint \qquad \fi=
  \left| I(G(u)) - \sum_{\setu \in \setU(\epsilon)} I_\setu(G(u_\setu^{h_\setu})) \right|^2
  +
  \sum_{\setu \in \setU(\epsilon)} \E_{\bsDelta_\setu}\left[ \left| \left(I_\setu-Q^{\bsDelta_\setu}_{\setu,n_\setu}\right)(G(u_\setu^{h_\setu})) \right|^2
  \right]
  .
\end{align*}
Note that the first part can be split up in the sum of the truncation and FE discretization error, while for the second part we expect the root-mean-square error to behave like~\eqref{eq:Qu-rate}.
Using the same arguments as in the previous section and picking $p=\infty$ and $q=1$ we then receive
\begin{align}\label{eq:MDFEM-randomized-bound}
  \notag
  &\E_\bsDelta\left[\left|I(G(u)) - Q_\epsilon^{\mathrm{MDFEM}}(G(u)) \right|^2\right]
  \\
  &\ifpreprint \qquad \fi\le
  \left(
  \left( \sup_{|\setu| < \infty} \gamma_\setu^{-1} \, \|G(u_\setu)\|_{H_{\alpha,0,\setu}} \right)
  \left( \sum_{\setu \notin \setU(\epsilon)} \gamma_\setu \, M_\setu \right)
  +
  \widetilde{C} \left( \sum_{\setu \in \setU(\epsilon)} 2^{|\setu|} \, h_\setu^\tau \right)
  \right)^2
  \\
  \notag
  &\ifpreprint \qquad \quad \else \;\; \fi+
  \left(
  \left( \sup_{|\setu| < \infty} \gamma_\setu^{-1} \, \|G(u_\setu^{h_\setu})\|_{H_{\alpha,0,\setu}} \right)
  \max_{\setu \in \setU(\epsilon)} \left( \frac{\ln(n_\setu)}{|\setu|} \right)^{\lambda_1 |\setu|}
  \left( \sum_{\setu \in \setU(\epsilon)} \frac{\gamma_\setu \, C_{\setu,\lambda} \, |\setu|^{\lambda_1 |\setu|}}{n_\setu^\lambda} \right)
  \right)^2
  .
\end{align}
Using $a^2 + b^2 \le (a + b)^2$ for $a, b \ge 0$ and taking the square root on both sides we obtain exactly the same expression for the root-mean-square error as in the right hand side of~\eqref{eq:MDFEM-deterministic-bound}.

\subsection{Choosing the weight parameters $\gamma_\setu$}

For both the deterministic and the randomized error bound, see \eqref{eq:MDFEM-deterministic-bound} and~\eqref{eq:MDFEM-randomized-bound}, we need to choose $\{ \gamma_j\}_{j \ge 1}$ such that both
\begin{align*}
  \sup_{|\setu| < \infty} \gamma_\setu^{-1} \, \|G(u_\setu)\|_{H_{\alpha,0,\setu}}
  &<
  \infty
  &\text{and}&&
  \sup_{|\setu| < \infty} \gamma_\setu^{-1} \, \|G(u_\setu^{h_\setu})\|_{H_{\alpha,0,\setu}}
  &<
  \infty
  .
\end{align*}
Applying~\RefLem{lem:norm-uu-Guu} we have
\begin{align*}
  \sup_{|\setu| < \infty} \gamma_\setu^{-1} \, \|G(u_\setu)\|_{H_{\alpha,0,\setu}}
  &\le
  \frac{C_{\kappa,\alpha}^{1/2} \, \|f\|_{V^*} \, \|G\|_{V^*}}{(1 - \kappa) \, a_{0,\min}}
\sup_{|\setu| < \infty} \gamma_\setu^{-1} \prod_{j \in \setu} b_j
  ,
\end{align*}
which is finite if we choose
\begin{align}\label{eq:product-weights}
  \gamma_j
  =
  b_j
  .
\end{align}
By \RefLem{lem:norm-uu-Guu} the same holds for $\|G(u_\setu^{h_\setu})\|_{H_{\alpha,0,\setu}}$.
Obviously this means $\{\gamma_j\}_{j\ge1} \in \ell^{p^*}(\N)$ since $\{b_j\}_{j\ge1} \in \ell^{p^*}(\N)$.

\begin{remark}\label{rem:Gu-well-defined}
With the choice of weights~\eqref{eq:product-weights} and under the conditions~\eqref{eq:a0-max-min},~\eqref{eq:kappa} and~\eqref{eq:bj-pstar-summable} the decomposition~\eqref{eq:G-MDM},
\begin{align*}
  G(u(\bsx, \bsy))
  =
  \sum_{|\setu| < \infty} G(u_\setu(\bsx, \bsy_\setu))
\end{align*}
is well-defined, i.e., for any $\bsy \in \Omega^\N$ and any $\bsx \in D$
\begin{align*}
  \left|\sum_{|\setu| < \infty} G(u_\setu(\bsx, \bsy_\setu))\right|
  <
  \infty
  .
\end{align*}
Indeed, using the reproducing property of $K_{\alpha,0,\setu}$ and the Cauchy--Schwarz inequality we have
\begin{align*}
  \left| \sum_{|\setu| < \infty} G(u_\setu(\bsx, \bsy_\setu)) \right|
  &=
  \left| \sum_{|\setu| < \infty} \langle G(u_\setu(\bsx, \cdot)), K_{\alpha,0,\setu}(\bsy_\setu, \cdot) \rangle_{H_{\alpha,0,\setu}}\right| 
  \\
  &\le
  \sum_{|\setu| < \infty} 
  \| G(u_\setu)\|_{H_{\alpha,0,\setu}} \, \|K_{\alpha,0,\setu}(\bsy_\setu, \cdot)\|_{H_{\alpha,0,\setu}}
  \\
  &=
  \sum_{|\setu| < \infty}
	\| G(u_\setu)\|_{H_{\alpha,0,\setu}} \, (K_{\alpha,0,\setu}(\bsy_\setu, \bsy_\setu))^{1/2}
  \\
  &\le
  \left( \sup_{|\setu| < \infty} \gamma_\setu^{-1} \, \|G(u_\setu)\|_{H_{\alpha,0,\setu}} \right)
  \left( \sum_{|\setu| < \infty} \gamma_\setu \, (K_{\alpha,0,\setu}(\bsy_\setu, \bsy_\setu))^{1/2} \right)
  .
  \end{align*}
The first term is finite due to the way we choose $\gamma_j$ as in~\eqref{eq:product-weights}.
For the second term we have
\begin{align*}
  \sum_{|\setu| < \infty} \gamma_\setu \, (K_{\alpha,0,\setu}(\bsy_\setu, \bsy_\setu))^{1/2}
  &\le
  \sum_{|\setu| < \infty} \gamma_\setu \, \max_{\bsy_\setu \in \Omega_\setu} |K_{\alpha,0,\setu}(\bsy_\setu, \bsy_\setu)|^{1/2}
  \le
  \sum_{|\setu| < \infty} \gamma_\setu \, M_\setu
   ,
\end{align*}
where $M_\setu = M^{|\setu|}$ and with $M$ given as in~\eqref{eq:H-alpha-0:M}.
Applying \RefLem{lem:sum-finite} we have $\sum_{|\setu| < \infty} \gamma_\setu \, M_\setu < \infty$, which implies the needed claim.
Note that we need to demand $p^* \le 1$ to apply \RefLem{lem:sum-finite} here.
\end{remark}

\subsection{Computational cost}\label{sec:computational-cost}

Now we study the computational cost of the proposed method~\eqref{eq:MDFEM}.
To obtain $u_\setu^{h_\setu}$, see~\eqref{eq:uu-hu}, we have to calculate $\setv$-truncated solutions, cf.~\eqref{eq:u-truncated-solution}, for each $\setv \subseteq \setu$, i.e., $2^{|\setu|}$ solutions of the PDE with meshwidth $h_\setu$.
Hence, for each node $\bsy_\setu^{(k)}$ of the $n_\setu$-point cubature method and for each $\setv \subseteq \setu$ the FEM leads to solve a system of linear equations.
Due to the locality of the polynomials of $V^{h_\setu}$ the matrix
is sparse and has $O(h_\setu^{-d})$ nonzero elements, where $d$ is the physical dimension, e.g., $d=1,2,3$.
We assume the cost of solving the sparse linear system is nearly linear, i.e., of order $O(h_\setu^{-d'})$ with $d' \sim d$, e.g., in~\cite{KSS2012,KNPSW2017,DKLNS2014} this cost was assumed to be linear with $d' = d$.
To evaluate each element of the matrix we assume that it is dominated by the cost to evaluate $a(\bsx,\bsy_\setv)$ and we bound this by~$O(|\setv|)$.
Thus, the cost for evaluating the stiffness matrix for the $\setv$-truncated solution is~$O(h_\setu^{-d} \, |\setv|)$, which we can estimate as $O(h_\setu^{-d} \, |\setu|)$ for every $\setv \subseteq \setu$.

As a result, the total computational cost of the MDFEM is given by
\begin{align}\label{eq:MDFEM:cost}
  \cost(Q_\epsilon^{\mathrm{MDFEM}})
  =
  O\!\left( \sum_{\setu \in \setU(\epsilon)} n_\setu \, h_\setu^{-d'} \, \pounds_\setu \right)
  ,
\end{align}
with $\pounds_\setu = 2^{|\setu|} |\setu|$ and $d' \sim d$.

Similar to \RefSec{sec:MDM}, the key idea of the MDFEM is to first select the active set $\setU(\epsilon)$ such that the truncation error is bounded by $\epsilon/2$ and then for every $\setu \in \setU(\epsilon)$ choose $h_\setu$ and the cubature rules $Q_{\setu,n_\setu}$  such that the computational cost~\eqref{eq:MDFEM:cost} is minimized with respect to the combination of the cubature error and the FE discretization error being bounded by $\epsilon/2$.
Since our spaces $H_{\alpha,0,\setu}$ are unweighted, the cubature rules can be reused and their construction (e.g., by constructing good generating vectors) can be considered as an a priori cost and therefore we do not include it in the total cost.

\subsection{Selection of the MDFEM active set}

Since we have now chosen $p=\infty$, $q=1$ and $\gamma_j = b_j$, we use \RefRem{rem:choice-of-rho} to minimize the size of the active set~\eqref{eq:active-set} for the MDFEM by using
\begin{align}\label{eq:MDFEM:active-set}
  \setU(\epsilon, 1, 1/p^*)
  =
  \left\{
    \setu
    :
    \left(\gamma_\setu \, M_\setu\right)^{(1-p^*)}
    >
    \frac{\epsilon/2}{\sum_{|\setv| < \infty} (\gamma_\setv \, M_\setv)^{p^*}}
  \right\}
  .
\end{align}

\subsection{Selection of the finite element and cubature approximations}

Similar to the optimization problem in \RefSec{sec:MDM:Lagrange}, we look for positive real numbers $k_\setu$ and $h_\setu$, and then set
\begin{align}\label{eq:MDFEM:nu}
  n_\setu
  &=
  2^{\log_2(\lfloor k_\setu \rfloor)} \in \N_0
  ,
\end{align}
such that $k_\setu$ and $h_\setu$ solve the following optimization problem:
\begin{align}\label{eq:MDFEM:optim}\begin{split}
  &\text{minimize } \sum_{\setu \in \setU(\epsilon,1,1/p^*)} k_\setu \, h_\setu^{-d'} \, \pounds_\setu
  \\
  &\text{subject to } \sum_{\setu \in \setU(\epsilon,1,1/p^*)} \left( \frac{\gamma_\setu \, 2^\lambda \, C_{\setu,\lambda} \, |\setu|^{\lambda_1|\setu|}}{k_\setu^\lambda} + 2^{|\setu|} \, h_\setu^\tau \right)
  =
  \frac{\epsilon}{2}
  ,
\end{split}\end{align}
with $\pounds_\setu = 2^{|\setu|} \, |\setu|$.
This can be solved using the Lagrange multiplier method.
We refer to \cite{NN2021} where this was worked out in the context of the MDFEM with a lognormal random field.
We obtain
\begin{align}
  \label{eq:MDFEM:ku}
  k_\setu
  &=
  \left(\frac{\epsilon}{2}\right)^{-1/\lambda}
  \left(\frac{\tau + \lambda \, d'}{\tau}\right)^{1/\lambda}
  \left(
    \frac{\gamma_\setu \, 2^\lambda \, C_{\setu,\lambda} \, |\setu|^{\lambda_1|\setu|}}{2^{|\setu| d'} \, \pounds_\setu^{\tau}}
  \right)^{1/(\tau + \lambda (\tau + d'))}
  K_\epsilon^{1/\lambda}
\intertext{and}
  \label{eq:MDFEM:hu}
  h_\setu
  &=
  \left(\frac{\epsilon}{2}\right)^{1/\tau}
  \left(\frac{\lambda \, d' \, 2^{-|\setu|}}{\tau + \lambda \, d'}\right)^{1/\tau}
  \left(\frac{\gamma_\setu \, 2^\lambda \, C_{\setu,\lambda} \, |\setu|^{\lambda_1 |\setu|} \, \pounds_\setu^\lambda}{2^{|\setu|(\lambda+1)}}\right)^{1/(\tau + \lambda (\tau+d'))}
  K_\epsilon^{-1/\tau}
\intertext{with}
  \notag
  K_\epsilon
  &:=
  \sum_{\setu \in \setU(\epsilon,1,q/p^*)} \left( \gamma_\setu^\tau \, 2^{\lambda\tau} \, C_{\setu_,\lambda}^\tau \, |\setu|^{\lambda_1 \tau |\setu|} \, 2^{\lambda d' |\setu|} \, \pounds_\setu^{\lambda \tau} \right)^{1/(\tau+\lambda(\tau+d'))}
  .
\end{align}
Again, making use of \RefLem{lem:sum-finite} to have $\lim_{\epsilon \to 0} K_\epsilon$ absolutely bounded we obtain the conditions $\lambda_1 < 1 + \lambda (1 + d'/\tau)$ and $\lambda \le (1 - p^*) / (p^* ( 1 + d'/\tau))$.
The same conditions also make the cost uniformly bounded and we can write
\begin{align*}
  \cost(Q_\epsilon^{\mathrm{MDFEM}})
  \lesssim
  \epsilon^{-1/\lambda - d'/\tau}
  .
\end{align*}
It is easy to see that bigger values of $\lambda$ give lower bounds for the computational cost, so in~\RefThm{thm:main-theorem} we will choose $\lambda$ as big as possible, i.e., 
$\lambda = (1-p^*) / (p^*(1+d'/\tau))$.

\subsection{Main result}

Finally combining the selection of the active set, the cubature rules and the finite element discretizations we obtain our main result.

\begin{theorem}\label{thm:main-theorem}
	Let $a_0 \in L^\infty(D)$ be such that $\essinf a_0 >0$, and assume there exists a sequence $\{b_j\}_{j\ge 1} \in \ell^{p^*}(\N)$ with $0< b_j \le 1$ for all $j$ and some $p^* \in (0,1)$, such that
	\begin{align*}
	  \kappa
	  =
	  \left\| \frac{\sum_{j \ge 1} |\phi_j|/b_j}{2 a_0} \right\|_{L^\infty(D)}
	  <
	  1
	  .
	\end{align*}
    Assume the used FEM converges with a rate $\tau$ as in~\eqref{eq:FE-error-Gu}, with the particular conditions~\eqref{eq:a-W1} or~\eqref{eq:a-Wt}, and solving the linear systems costs $O(h_\setu^{d'})$.
	Let, for a given requested error tolerance $\epsilon > 0$, the active set $\setU(\epsilon, 1, 1/p^*)$ be chosen as in~\eqref{eq:MDFEM:active-set},
	the number of cubature points $n_\setu$
	be chosen as in~\eqref{eq:MDFEM:nu}, with $k_\setu$ given by~\eqref{eq:MDFEM:ku},
	and the meshwidths $h_\setu$ 
	be chosen as in~\eqref{eq:MDFEM:hu}, i.e., as the solution to the optimization problem~\eqref{eq:MDFEM:optim}.
	Then with $\alpha = \left\lfloor\frac{\tau(1-p^*)}{p^*(\tau+d')}\right\rfloor+1$ and for $\kappa < \frac{1}{2\alpha+1}$ the following hold.
	\begin{enumerate}
		\item If $\frac{\tau(1-p^*)}{p^*(\tau+d')} \ge 1$ then the MDFEM based on interlaced polynomial lattice rules with interlacing factor $\alpha$
		and convergence as in~\eqref{eq:IPLR} with $\lambda = \frac{\tau(1-p^*)}{p^*(\tau+d')}$ achieves
		\begin{align*}
		  \left|I(G(u)) - Q_\epsilon^{\mathrm{MDFEM}}(G(u)) \right|
		  \lesssim
		  \epsilon
		  .
		\end{align*}
		\item If $\frac1{2} \le \frac{\tau(1-p^*)}{p^*(\tau+d')} < 1$ then the MDFEM  based on randomly digitally shifted polynomial lattice rules with convergence as in~\eqref{eq:RPLR} with $\lambda = \frac{\tau(1-p^*)}{p^*(\tau+d')}$ achieves
		\begin{align*}
		  \sqrt{\E_\bsDelta\left[\left|I(G(u)) - Q_\epsilon^{\mathrm{MDFEM}}(G(u))\right|^2\right]}
		  \lesssim
		  \epsilon
		  .
		\end{align*}
	\end{enumerate}
	In both cases the computational cost is bounded as
	\begin{align*}
	  \cost(Q_\epsilon^{\mathrm{MDFEM}})
	  &\lesssim
	  \epsilon^{-a_{\mathrm{MDFEM}}}
	  &\text{with}&&
	  a_{\mathrm{MDFEM}}
	  :=
	  \frac{1}{\lambda} + \frac{d'}{\tau}
	  =
	  \frac{1+d'/\tau}{1/p^*-1} + \frac{d'}{\tau}
	  .
	\end{align*}
\end{theorem}

We now compare the MDFEM presented in this paper with the single-level quasi-Monte Carlo finite element method (SLQMCFEM) developed in \cite{GHS2018} which is a truncation algorithm for the parameters $y_j$ to some dimension~$s$.
The SLQMCFEM achieves an error, see~\cite[Equation~(38)]{GHS2018},
\begin{align*}
  \error(Q^{\mathrm{SLQMCFEM}})
  \lesssim
  n^{-1/p^*} + h^\tau+ \left(\sup_{j \ge s+1} \{b_j\}\right)^2
  ,
\end{align*}
where $n$ is the number of cubature points, $h$ is the finite element meshwidth and $s$ is the truncation dimension.
Assume a similar computational cost setting as in~\RefSec{sec:computational-cost}, i.e.,
\begin{align*}
  \cost(Q^{\mathrm{SLQMCFEM}})
  \lesssim
  n \, h^{-d'} \, s
  .
\end{align*}
To achieve an error of order $O(\epsilon)$ the computational cost of the SLQMCFEM is of order $O(\epsilon^{-a_{\mathrm{SL}}})$ with $a_{\mathrm{SL}} := d'/\tau + 3p^*/2$.
Hence, we have
\begin{align*}
  a_{\mathrm{SL}} - a_{\mathrm{MDFEM}}
  =
  p^*\left(\frac{3}{2} - \frac1{1-p^*} - \frac{d'}{\tau(1-p^*)}\right)
\end{align*}
which is positive when $d'/\tau+ 3p^*/2 < 1/2$. This means that the MDFEM outperforms the SLQMCFEM when $p^* < 1/3 - 2d'/(3\tau)$, i.e., when the terms in the expansion of the diffusion coefficient decay sufficiently fast. 

We note that the cost model in~\cite{GHS2018} takes advantage of the wavelet decomposition to obtain a discretization of the random field, but also in that case the MDFEM can outperform the SLQMCFEM when $p^*$ is small enough. It is likely that also the MDFEM can take advantage of the wavelet decomposition, but it is not immediately clear how to incorporate this into the cost analysis.

\section{Conclusion and further work}\label{sec:conclusion}

In this work we have proposed the MDFEM which is an algorithm combining the MDM with the FEM and have applied it to elliptic PDEs with uniform random diffusion coefficients. We have analyzed the error and the computational cost of the proposed method. It has been theoretically shown that our method is competitive with SLQMCFEM in term of error versus computational cost.

We give some further remarks on implementing the MDFEM.
Once the active set of the MDFEM is selected, the different parts of the decomposed form can be computed in parallel.
Moreover, because of the recursive structure of the anchored decomposition there is a chance to save computational cost by reducing the number of repeated function evaluations.
Such a method has been analyzed in~\cite{GKNW2018}.

The general MDM is shown to be efficient for infinite-dimensional integrals with respect to general probability measures, and it is capable of retrieving a convergence rate very close to that of the used cubature rules for the finite-dimensional integrals. The analysis in this paper for the MDFEM is restricted to uniform diffusion coefficients, i.e., to integrals with respect to uniform distributions.
This analysis has been extended to log-normal diffusion coefficients, that is, when $a(\bsx,\bsy) = \exp(Z(\bsx,\bsy))$ where $Z$ is a Gaussian random field in~\cite{NN2021}.

\bibliographystyle{plain}
\bibliography{MDM}

\end{document}